\documentclass[reqno]{amsart}
\usepackage[english]{babel}
\usepackage{amscd,amssymb,amsmath,amsfonts,latexsym,amsthm}
\usepackage{inputenc}
\usepackage{graphicx, tikz-cd}
\usepackage{cite}

\numberwithin{equation}{section}

\numberwithin{equation}{section}

\newtheorem{thm}{Theorem}[section]
 \newtheorem{cor}[thm]{Corollary}
 
 \newtheorem{prop}[thm]{Proposition}
 \newtheorem{defn}[thm]{Definition}

\def\vtr{\vartriangleright}
\def\vtl{\vartriangleleft}

\begin{document}
\title[Classification of four-dimensional anti-dendriform algebras]{Classification of four-dimensional anti-dendriform algebras whose
associated associative algebra has the center of dimension one}
		
\author{Jobir Adashev, Artem Lopatin,  Zafar Normatov, and Shokhsanam Solijonova}

\address[Jobir Adashev]{Institute of Mathematics, Uzbekistan Academy of Sciences, Univesity Street, 9, Olmazor district, Tashkent, 100174, Uzbekistan}
\email{adashevjq@mail.ru}

\address[Artem Lopatin]{Universidade Estadual de Campinas (UNICAMP),  651 Sergio Buarque de Holanda, 13083-859 Campinas, SP, Brazil}
\email{dr.artem.lopatin@gmail.com}

\address[Zafar Normatov]{School of Mathematics, Jilin University, Changchun, 130012, China}
\email{z.normatov@mathinst.uz}

\address[Shokhsanam Solijonova]{	
National University of Uzbekistan, Univesity Street, 4, Olmazor district, Tashkent, 100174, Uzbekistan}
\email{sh.solijonova@mail.ru}

\begin{abstract} This article is devoted to the classification of anti-dendriform algebras that are associated with associativity. They are characterized as algebras with two operations whose sum is associative. In the paper all four-dimensional complex anti-dendriform algebras associated to four-dimensional associative algebras with one-dimensional center are classified.
\end{abstract}

\subjclass[2020]{16P10, 17A30. }
\keywords{associative algebra; anti-dendriform algebra.}
	
	\maketitle

\section{Introduction}
\

The idea of considering algebras with two or more multiplications is not new, but it is under a certain consideration. So,  associative dialgebras as vector spaces equipped with two associative operations were considered in \cite{Ch,Fra,U}.  A dual notion of the Poisson algebra by exchanging the roles of the two binary operations in the Leibniz rule defining the Poisson algebra was introduced by Bai, Bai, Guo, and Wu (see, references in \cite{aae24,sar24}).   In \cite{f25} considered a new family of modified double Poisson brackets and mixed double Poisson algebras. 
 In paper \cite{GP} authors defined a dendriform di- or trialgebra in an arbitrary variety Var of binary algebras (associative, commutative, Poisson, etc.) and proved that every dendriform dialgebra can be embedded into a Rota--Baxter algebra of weight zero in the same variety, and every dendriform trialgebra can be embedded into a Rota--Baxter algebra of nonzero weight. In these papers \cite{b,Khr,Le1,LR,RRB} also study algebras that are defined by several multiplications and require satisfying identities.

The notion of anti-dendriform algebras as a new approach of splitting the associativity is introduced in \cite{DGC}. Anti-dendriform algebras are characterized as algebras with two operations whose sum is associative and the negative left and right multiplication operators compose the bimodules of the sum associative algebras, justifying the notion due to the comparison with the corresponding characterization of dendriform algebras.

Note that there is an anti-structure for pre-Lie algebras, namely anti-pre-Lie algebras, introduced in \cite{LB}, which are characterized as the Lie-admissible algebras. There is a new approach to splitting operations, motivated by the study of anti-pre-Lie algebras. We introduce the notion of anti-dendriform algebras, still keeping the property of splitting the associativity, but it is the negative left and right multiplication operators that compose the bimodules of the sum associative algebras, instead of the left and right multiplication operators doing so for dendriform algebras. Such a characterization justifies the notion, and the following commutative diagram holds, which is the above diagram with replacing dendriform and pre-Lie algebras by anti-dendriform and anti-pre-Lie algebras respectively.

$$\begin{matrix} \mbox{ anti-dendriform algebras} & \longrightarrow & \mbox{anti-pre-Lie algebras} \cr \downarrow & &\downarrow\cr {\rm
associative\quad algebras} & \longrightarrow & {\rm Lie\quad
algebras} \cr\end{matrix}$$

The classification of any class of algebras is a fundamental and very difficult problem. It is one of the first problems that one encounters when trying to understand the structure of this class of algebras. This paper is devoted to classifying anti-dendriform algebras associated with null-filiform associative algebras and three-dimensional algebras. The algebraic classification (up to isomorphism) of algebras of dimension $n$ from a certain variety defined by a certain family of polynomial identities is a classic problem in the theory of non-associative algebras. There are many results related to the algebraic classification of small-dimensional algebras in many varieties of
associative and non-associative algebras \cite{BT22,ak21,afk21,EM22,CKLS,fkk21,fkkv22,Kay,mk22}. The algebraic classification gives a way to obtain the geometric classification \cite{MS}.
Algebraic classifications of $2$-dimensional algebras \cite{petersson}, $3$-dimensional evolution algebras \cite{ccsmv}, $3$-dimensional anticommutative algebras \cite{Kobayashi,japan}, $3$-dimensional diassociative algebras \cite{RIB} and classification of non-isomorphic complex $3$-dimensional transposed Poisson algebras \cite{BOK},  have been given.

The paper is organized as follows. In Section 2, we give necessary  definitions and formulate some preliminary results. In particular, to every anti-dentriform algebra in a unique way is associated an associative algebra, which is known to be nilpotent (see Proposition~\ref{prop_2.2} and Corollary~\ref{cor_2.5}). Then a classification of all four-dimensional complex nilpotent indecomposible associative algebra is formulated in Theorem~\ref{table}, which was proven in~\cite{BRR}. In Section 3 we obtain a classification of all four-dimensional anti-dendriform algebras associated with four-dimensional associative algebras which one-dimensional center (see Theorems~\ref{theo_3.1}, \ref{theo_3.2}, \ref{theo_3.3}, \ref{theo_3.4}, \ref{theo_3.5}, \ref{theo_3.6}, \ref{theo_3.7}, \ref{theo_3.8}).

\section{Preliminaries}

Throughout the paper, all vector spaces and algebras
are finite-dimensional and over the complex field $\mathbb{C}$ unless otherwise stated. Let $A$ be a vector space with two bilinear operations
\[
\vtr: A\otimes A \rightarrow A, \qquad \vtl: A\otimes A \rightarrow A.
\]

Define a bilinear operation $\cdot$ as follows:
\begin{equation}\label{cdot}
x \cdot y=x \vtr y+ x\vtl y, \quad \forall \ x, y \in A.
\end{equation}
If $(A,\cdot)$ is an associative algebra, then we call the triple $(A, \vtr, \vtl)$ an \textit{associative admissible algebra} and  $(A,\cdot)$ \textit{the associated associative algebra of} $(A, \vtr, \vtl)$.

\begin{defn}[Loday, Frabetti, Chapoton,  Goichot~\cite{Cha1}]
Let $A$ be a vector space with two bilinear operations $\vartriangleright$ and $\vartriangleleft$. The triple $(A, \vtr, \vtl)$ is called an anti-dendriform algebra if the following equations hold:
\begin{align}\label{id1}
(x\vtr y)\vtl z=x\vtr(y\vtl z),\\\label{id2}
x\vtr (y\vtr z)=-(x\cdot y)\vtr z,\\\label{id3}
x\vtr (y\vtr z)=-x\vtl(y\cdot z),\\\label{id4}
x\vtr (y\vtr z)=(x\vtl y)\vtl z,\\\label{id5}
(x\cdot y)\vtr z=x\vtl(y\cdot z),\\\label{id6}
-(x\cdot y)\vtr z=(x\vtl y)\vtl z,\\\label{id7}
-x\vtl (y\cdot z)=(x\vtl y)\vtl z.
\end{align}
\end{defn}

\begin{prop}[\cite{DGC}]\label{prop_2.2}
Let $(A, \vtr, \vtl)$ be an anti-dendriform algebra. Define a bilinear operation $\cdot$ by \eqref{cdot}.
Then $(A, \cdot)$ is an associative algebra, called
\textbf{the associated associative algebra} of $(A, \vtr, \vtl)$. Furthermore, $(A, \vtr, \vtl)$ is called
\textbf{a compatible anti-dendriform algebra structure} on $(A, \cdot)$.
\end{prop}

Recall that an associative algebra $(A, \cdot)$ is 2-nilpotent if $(x\cdot y)\cdot z = x\cdot (y\cdot z) = 0$ for all $x, y, z \in A$. Similarly, an anti-dendriform algebra $(A,\vtr,\vtl)$ is 2-nilpotent if $(x\ast_{i_1} y)\ast_{i_2} z=x\ast_{i_3} (y\ast_{i_4} z)=0$ for all $i_1, i_2, i_3, i_4\in\{\vtr,\vtl\}$ and $x,y,z\in A$. Using equations  (\ref{id1})-(\ref{id7}), it was  shown in \cite{aans25}  that the anti-dendriform algebras associated with the associative abelian algebra are 2-nilpotent.

The following sets are called the center of associative and anti-dendriform algebras, respectively: $${\rm
Z}_{As}(A)=\{x\in A~|~x\cdot y=y\cdot x=0,\ \forall y\in A\},$$
$${\rm
Z}_{AD}(A)=\{x\in A~|~x\vtr y=x\vtl y=y\vtr x=y\vtl x=0,\ \forall y\in A\}.$$

An \textit{ideal} $I$ of an anti-dendriform algebra $A$ is a subalgebra of the algebra $A$ that satisfies the conditions:
$$x\vtr y, \ x\vtl y, \ y\vtr x, \ y\vtl x \in I, \ \ \mbox{for all} \ \  x\in A,\  y\in I.$$
It is obvious that center of an arbitrary anti-dendriform algebra is an ideal.

\begin{prop}[\cite{aans25}]\label{compatible}
Let $(A,\cdot)$ be an associative algebra and let $(A, \vtr, \vtl)$ be a compatible anti-dendriform algebra structure on $(A,\cdot)$. If the center $(Z(A), \cdot)$ of $(A, \cdot)$ is the center $(Z(A), \vtr, \vtl)$ of $(A, \vtr, \vtl)$, then the quotient $(A/Z(A), \vtr, \vtl)$ is a compatible anti-dendriform algebra structure on $(A/Z(A), \cdot)$.
\end{prop}

From now we use the following notations: $As_n^q$ and $AD_n^q$ denote $n$-dimensional associative and anti-dendriform algebra structures associated with nilpotent associative algebras, where $q$ is an index for numbering.

\begin{prop}[\cite{DGC}]\label{idempotent}
Let $(A, \cdot)$ be an associative algebra with a non-zero idempotent $e$, that is, $e \cdot e= e$. Then there does not exist a compatible anti-dendriform algebra structure on $(A, \cdot)$.
\end{prop}

It is known that any finite-dimensional associative algebra without a non-zero idempotent element is nilpotent. Therefore we have the following conclusion.

\begin{cor}\label{cor_2.5}
The associated associative algebra of any anti-dendriform algebra is nilpotent.
\end{cor}

For an algebra $A$ of an arbitrary variety, we consider the series
$$A^1=A, \ \ A^{i+1}=\sum\limits_{k=1}^{i}A^kA^{i+1-k}, \ \ i\geq1.$$
We say that an algebra $A$ is \textit{nilpotent} if $A^i =\{0\}$ for some $i\in \mathbb{N}.$ The smallest integer satisfying $A^i = \{0\}$
is called the \textit{index of nilpotency} of $A.$ Obviously, for an $n$-dimensional nilpotent algebra $A$ we have $A^{n+1} = \{0\}$, i.e. the index of nilpotency of an $n$-dimensional nilpotent algebra is not greater than $n + 1$. An $n$-dimensional algebra $A$ is called \textit{maximum nilpotent} if the  nilpotency index of the algebra is equal to $n+1$.

By Proposition \ref{idempotent}, in order to classify arbitrary three-dimensional anti-dendriform algebras over a complex number field, we present a classification of all complex three-dimensional nilpotent associative algebras.

\begin{thm}[\cite{Kobayashi}]\label{assdim3}
Any three-dimensional complex  nilpotent associative algebra is
isomorphic to one of the following pairwise non-isomorphic
associative algebras:

$As_3^1:\ Abelian;$

$As_3^2:\ e_1e_2=e_3, \ e_2e_1=-e_3;$

$As_3^3:\ e_1e_1=e_3;$

$As_3^4:\ e_1e_2=e_3;$

$As_3^5(\lambda):\ e_1e_1=e_3, \ e_1e_2=\lambda e_3, \ e_2e_2=e_3, \ \lambda\in \mathbb{C};$

$As_3^6:\ e_1e_1=e_2, \ e_1e_2=e_3, \ e_2e_1=e_3.$

\end{thm}

Now we formulate the classification of all three-dimensional anti-dendriform algebras corresponding to these three-dimensional associative algebras.

\begin{thm}[\cite{aans25}]  Any three-dimensional complex anti-dendriform algebra is isomorphic to one of the following pairwise non-isomorphic algebras:

$AD_3^1: \    e_1\vtr e_1=\frac12e_2, \ e_1\vtl e_1=\frac12e_2,\ e_1\vtr e_2=e_2\vtl e_1=2e_3,\ e_2\vtr e_1=e_1\vtl e_2=-e_3;$

$AD_3^2: \ e_1\vtr e_1=\frac12e_2+e_3,\ e_1\vtl e_1=\frac12e_2-e_3,\ e_1\vtr e_2=e_2\vtl e_1=2e_3, \ e_2\vtr e_1=e_1\vtl e_2=-e_3;$

$AD_3^3: \ \mbox{is trivial, that is, all products are zero;}$

$AD_3^4: \ e_1\vtr e_2=e_3, \ e_2\vtr e_1=-e_3,\ e_1\vtl e_2=-e_3, \ e_2\vtl e_1=e_3;$

$AD_3^5: \ e_1\vtr e_1=e_3,\ e_1\vtl e_1=-e_3;$

$AD_3^6: \ e_1\vtr e_2=e_3, \ e_1\vtl e_2=-e_3;$

$AD_3^7(\lambda): \ \begin{cases}
e_1\vtr e_1=e_3,\ e_1\vtr e_2=\lambda e_3, \ e_2\vtr e_2=e_3,\\
e_1\vtl e_1=-e_3,\ e_1\vtl e_2=-\lambda e_3, \ e_2\vtl e_2=-e_3;
\end{cases}
$

$AD_3^8(\alpha,\beta): \ \begin{cases}
e_1\vtr e_1=e_3, \ e_1\vtr e_2=\alpha e_3, \
e_2\vtr e_1=\beta e_3,\\[1mm]
e_1\vtl e_1=-e_3, \  e_1\vtl e_2=(1-\alpha)e_3, \
e_2\vtl e_1=(-1-\beta)e_3;
\end{cases}$

$AD_3^9(\alpha): \ e_1\vtr e_2=\alpha e_3, \ e_2\vtr e_1=-\alpha e_3, \ e_1\vtl e_2=(1-\alpha)e_3, \ e_2\vtl e_1=(-1+\alpha)e_3;
$

$AD_3^{10}: \ e_1\vtr e_1=e_2, \ e_2\vtr e_1=-e_3, \ e_1\vtl e_1=-e_2, \ e_1\vtl e_2=e_3;$

$AD_3^{11}(\alpha,\beta): \
e_1\vtr e_2=\alpha e_3,\
e_2\vtr e_1=\beta e_3,\
e_1\vtl e_2=(1-\alpha)e_3,\
e_2\vtl e_1=-\beta e_3;
$

$AD_3^{12}(\alpha,\beta): \ \begin{cases}
e_1\vtr e_2=\alpha e_3,\
e_2\vtr e_1=\beta e_3,\
e_2\vtr e_2=e_3,\\
e_1\vtl e_2=(1-\alpha)e_3,\
 e_2\vtl e_1=-\beta e_3,\
e_2\vtl e_2=-e_3;
\end{cases}$

$ AD_3^{13}(\alpha,\beta,\gamma): \ \begin{cases}
e_1\vtr e_1=e_3, \
e_1\vtr e_2=\alpha e_3,\
e_2\vtr e_1=\beta e_3,\
e_2\vtr e_2=\gamma e_3,\\
e_1\vtl e_1=-e_3,\
e_1\vtl e_2=(1-\alpha)e_3,\
e_2\vtl e_1=-\beta e_3,\
e_2\vtl e_2=-\gamma e_3;
\end{cases}$

$AD_3^{14}:\ e_1\vtr e_1=e_2, \ e_1\vtl e_1=-e_2, \ e_1\vtl e_2=e_3;$

$AD_3^{15}(\alpha,\beta,\gamma,\lambda):\ \begin{cases}
e_1\vtr e_1=\alpha e_3,\
e_1\vtr e_2=\beta e_3,\
e_2\vtr e_1=\gamma e_3, \  \alpha\neq 0,\\
e_1\vtl e_1=(1-\alpha)e_3,\
e_1\vtl e_2=(\lambda-\beta)e_3,\
e_2\vtl e_1=-\gamma e_3,\
e_2\vtl e_2=e_3;
\end{cases}$

$AD_3^{16}(\alpha, \lambda):\ \begin{cases}
e_1\vtr e_2=\alpha e_3,\
e_2\vtr e_1=-\alpha e_3,\\
e_1\vtl e_1=e_3,\
e_1\vtl e_2=(\lambda-\alpha)e_3,\
e_2\vtl e_1=-\alpha e_3,\
e_2\vtl e_2=e_3;
\end{cases}$

$AD_3^{17}(\lambda): \ e_1\vtr e_1=e_2,\
e_1\vtl e_1=-e_2+e_3,\
e_1\vtl e_2=\lambda e_3,\ e_2\vtl e_2=e_3;
$

$AD_3^{18}(\alpha): \ e_1\vtr e_1=\alpha e_3, \ e_1\vtl e_1=(1-\alpha)e_3;$

$AD_3^{19}: \ e_2\vtr e_1=e_3,\ e_1\vtl e_1=e_3, \ e_2\vtl e_1=-e_3;$

$AD_3^{20}(\alpha): \ \begin{cases}
e_1\vtr e_1=\alpha e_3, \ e_1\vtr e_2=e_3, \ e_2\vtr e_1=-e_3,\\
e_1\vtl e_1=(1-\alpha)e_3, \
e_1\vtl e_2=-e_3, \
e_2\vtl e_1=e_3;
\end{cases}$

$AD_3^{21}(\alpha):\  \begin{cases}
e_1\vtr e_2=e_3, \ e_2\vtr e_1=\alpha e_3,\\
e_1\vtl e_1=e_3, \
e_1\vtl e_2=-e_3, \
e_2\vtl e_1=-\alpha e_3, \ \alpha\neq-1;
\end{cases}$

$AD_3^{22}(\alpha,\beta): \ \begin{cases}
e_1\vtr e_1=\alpha e_3,  \ e_2\vtr e_1=\beta e_3, \  e_2\vtr e_2=e_3,\\
e_1\vtl e_1=(1-\alpha)e_3, \
e_2\vtl e_1=-\beta e_3, \
e_2\vtl e_2=-e_3;
\end{cases}$

$AD_3^{23}:\ e_1\vtr e_1=e_2,\ e_1\vtl e_1=-e_2+e_3;$\\
where  $\alpha, \beta, \gamma, \lambda\in \mathbb{C}$ and $AD_3^8(\alpha,\beta)\cong AD_3^8(-\beta,-\alpha);$ $AD_3^{15}(\alpha,\beta,\gamma,0)\cong AD_3^{15}(\alpha,-\beta,-\gamma,0);$ $AD_3^{21}(-1)\cong AD_3^{20}(0).$
\end{thm}

In this theorem, it can be seen the following:

\begin{itemize}
  \item the anti-dendriform algebras $AD_3^1$ and $AD_3^2$
 are compatible anti-dendriform algebra structures on $As_3^6;$
  \item the anti-dendriform algebras $AD_3^3-AD_3^7(\lambda)$
 are compatible anti-dendriform algebra structures on $As_3^1;$
  \item the anti-dendriform algebras $AD_3^8(\alpha,\beta)-AD_3^{10}$
 are compatible anti-dendriform algebra structures on $As_3^2;$
  \item the anti-dendriform algebras $AD_3^{11}(\alpha, \beta)-AD_3^{14}$
 are compatible anti-dendriform algebra structures on $As_3^4;$
  \item the anti-dendriform algebras $AD_3^{15}(\alpha, \beta,\gamma,\lambda)-AD_3^{17}(\lambda)$
 are compatible anti-dendriform algebra structures on $As_3^5(\lambda);$
  \item the anti-dendriform algebras $AD_3^{18}(\alpha)-AD_3^{23}$
 are compatible anti-dendriform algebra structures on $As_3^3.$
\end{itemize}

In order to classify arbitrary four-dimensional antidendriform algebras over the field of complex numbers we present a classification of all complex four-dimensional nilpotent associative algebras. An algebra is called decomposable if it is equal to a direct sum of two subalgebras such that each of them have a positive dimension. If an associative algebra is nilpotent and decomposable, then its center has dimension more than one. 

\begin{thm}(\cite{BRR}, see also~\cite{RRB_2009})\label{table}
Any four-dimensional complex nilpotent indecomposible associative algebra can be included in one of the following isomorphism classes of algebras:

\begin{footnotesize}
\begin{tabular}{|l|p{1.5in}|p{2.9in}|r|} \hline
\textit{Algebra}& \textit{Table of multiplication}& \textit{Automorphisms}& \textit{Center} \\
\hline
$As_{4}^{2}$ &$e_{1}e_{2}=e_{3},$ $e_{2}e_{1}=e_{4}$ & $\varphi_1(e_1)=ae_1+be_3+de_4,$ $\varphi_1(e_2)=ae_2+ce_3+ee_4,$ \newline $\varphi_1(e_3)=a^2e_3,$ $\varphi_1(e_4)=a^2e_4,$\newline $\varphi_2(e_1)=be_2+ce_3+ce_4,$ $\varphi_2(e_2)=ae_1+de_3+fe_4,$ \newline $\varphi_2(e_3)=abe_4,$ $\varphi_2(e_4)=abe_3,$ &$\langle e_3,e_4\rangle$\\
\hline
$As_{4}^{3}$ &$e_{1}e_{2}=e_{4},$ $e_{3}e_{1}=e_{4}$ & $\varphi(e_1)=ae_1+ce_2-ce_3+de_4,$ $\varphi(e_2)=be_2+ee_4,$ \newline $\varphi(e_3)=be_3+fe_4,$ $\varphi(e_4)=abe_4,$ &$\langle e_4\rangle$ \\ \hline
$As_{4}^{4}$ & $e_1e_2=e_3$, $e_2e_1=e_4$, \newline$e_2e_2=-e_3$ & $\varphi(e_1)=ae_1+be_3+de_4,$ $\varphi(e_2)=ae_2+ce_3+ee_4,$ \newline $\varphi(e_3)=a^2e_3,$ $\varphi(e_4)=a^2e_4,$&$\langle e_3,e_4\rangle$ \\ \hline
$As_{4}^{5}$& $e_1e_2=e_3,$ $e_2e_2=e_4$, \newline $e_2e_1=-e_3$ & $\varphi(e_1)=ae_1+de_3+fe_4,$ $\varphi(e_2)=ce_1+be_2+ee_3+ge_4,$ \newline $\varphi(e_3)=abe_3,$ $\varphi(e_4)=b^2e_4,$ &$\langle e_3,e_4\rangle$\\
\hline
$As_{4}^{6}$& $e_1e_2=e_4,$ $e_3e_3=e_4,$ \newline $e_2e_1=-e_4$ &$\varphi_1(e_1)=ae_1+ce_2+de_4,$ $\varphi_1(e_2)=-\frac{b^2}{c}e_1+ee_4,$ \newline $\varphi_1(e_3)=be_3+fe_4,$ $\varphi_1(e_4)=b^2e_4$ \newline
 $\varphi_2(e_1)=\frac{d^2+ab}{c}e_1+be_2+ee_4,$ $\varphi_2(e_2)=ae_1+ce_2+fe_4,$ \newline $\varphi_2(e_3)=de_3+he_4,$ $\varphi_2(e_4)=d^2e_4,$&$\langle e_4\rangle$\\ \hline
$As_{4}^{7}(\alpha)$ &$e_1e_2=e_4$, $e_2e_2=e_3,$ \newline $e_2e_1=\frac{1+\alpha}{1-\alpha}e_4$, $\alpha\neq 1$ & $\varphi(e_1)=ae_1+de_3+fe_4,$ $\varphi(e_2)=ce_1+be_2+ee_3+ge_4,$ \newline $\varphi(e_3)=b^2e_3+bc(1+\alpha)e_4,$ $\varphi(e_4)=abe_4,$ &$\langle e_3,e_4\rangle$\\\hline
$As_{4}^{8}$ & $e_1e_1=e_3,$ $e_1e_3=e_4,$ \newline $e_2e_2=-e_4,$ $e_3e_1=e_4$ & $\varphi(e_1)=\sqrt[3]{a^2}e_1+be_2+ce_3+de_4,$ \newline $\varphi(e_2)=ae_2+b\sqrt[3]{a}e_3+ee_4,$ \newline $\varphi(e_3)=\sqrt[3]{a^4}e_3+(2c\sqrt[3]{a^2}-b^2)e_4,$ $\varphi(e_4)=a^2e_4,$ &$\langle e_4\rangle$\\ \hline
$As_{4}^{9}$ & $e_1e_1=e_3,$ $e_2e_1=e_4,$ \newline $e_1e_3=-e_4,$ $e_3e_1=-e_4$ & $\varphi(e_1)=ae_1+be_2+ce_3+ee_4,$ $\varphi(e_2)=a^2e_2+de_4,$\newline $\varphi(e_3)=a^2e_3+a(b-2c)e_4,$ $\varphi(e_4)=a^3e_4,$ &$\langle e_4\rangle$\\ \hline
$As_{4}^{10}$ & $e_1e_1=e_4,$  $e_1e_2=e_4,$ \newline $e_2e_1=-e_4,$  $e_3e_3=e_4$ & $\varphi_1(e_1)=ae_1+be_4,$ $\varphi_1(e_2)=ae_2+ce_4,$ \newline $\varphi_1(e_3)=ae_3+de_4,$ $\varphi_1(e_4)=a^2e_4,$\newline $\varphi_2(e_1)=ae_1+be_2+de_4,$ $\varphi_2(e_2)=\frac{c^2}{a}e_2+ee_4,$ \newline $\varphi_2(e_3)=ce_3+fe_4,$ $\varphi_2(e_4)=c^2e_4,$ &$\langle e_4\rangle$\\ \hline
$As_{4}^{11}$& $e_1e_1=e_4,$  $e_1e_2=e_3,$ \newline $e_2e_1=-e_3,$ $e_2e_2=-2e_3+e_4$ & $\varphi_1(e_1)=ae_1+be_3+de_4,$ $\varphi_1(e_2)=ae_2+ce_3+ee_4,$ \newline $\varphi_1(e_3)=a^2e_3,$ $\varphi_1(e_4)=a^2e_4,$\newline $\varphi_2(e_1)=ae_2+be_3+de_4,$ $\varphi_2(e_2)=ae_1+ce_3+ee_4,$ \newline $\varphi_2(e_3)=-a^2e_3,$ $\varphi_2(e_4)=-2a^2e_3+a^2e_4,$ &$\langle e_3,e_4\rangle$\\ \hline
$As_{4}^{12}(\alpha)$& $e_1e_1=e_4,$  $e_1e_2=e_3,$ \newline $e_2e_1=-\alpha e_4,$  $e_2e_2=-e_3$ & $\varphi_1(e_1)=ae_1+be_3+de_4,$ $\varphi_1(e_2)=ae_2+ce_3+ee_4,$ \newline $\varphi_1(e_3)=a^2e_3,$ $\varphi_1(e_4)=a^2e_4,$\newline $\varphi_2(e_1)=ae_2+be_3+de_4,$ $\varphi_2(e_2)=ae_1+ce_3+ee_4,$ \newline $\varphi_2(e_3)=-a^2e_3,$ $\varphi_2(e_4)=-2a^2e_3+a^2e_4,$  &$\langle e_3,e_4\rangle$\\ \hline
$As_{4}^{13}$ & $ e_1e_1=e_3,$ $e_1e_3=-e_4,$ \newline $e_2e_1=e_4,$  $e_2e_2=e_4,$ \newline $e_3e_1=-e_4$ & $\varphi(e_1)=e_1+ae_2+be_3+ce_4,$ $\varphi(e_2)=e_2+ae_3+de_4,$\newline $\varphi(e_3)=e_3+(a^2+a-2b)e_4,$ $\varphi(e_4)=e_4,$ &$\langle e_4\rangle$\\ \hline
$As_{4}^{14}$ & $e_1e_2=e_4,$  $e_1e_3=e_4,$ \newline $e_2e_1=-e_4,$   $e_2e_2=e_4,$\newline  $e_3e_1=e_4$ & $\varphi(e_1)=ae_1+be_2-\frac{b^2}{2a}e_3+ce_4,$ $ \varphi(e_2)=ae_2-be_3+de_4,$\newline $\varphi(e_3)=ae_3+ee_4,$ $\varphi(e_4)=a^2e_4,$ &$\langle e_4\rangle$\\\hline
$As_{4}^{15}(\alpha)$ & $e_1e_1=e_4,$  $e_1e_2=\alpha e_4,$ \newline $e_2e_1=-\alpha e_4,$  $e_2e_2=e_4,$\newline $e_3e_3=e_4$ & $\varphi(e_1)=ae_1+be_2+ce_3+de_4,$ $ \varphi(e_2)=ee_1+fe_2+ge_3+he_4,$\newline $\varphi(e_3)=ie_1+je_2+ke_3+le_4,$ $\varphi(e_4)=(af-be)e_4,$ &$\langle e_4\rangle$\\ \hline
$As_{4}^{16}$ & $e_1e_1=e_2,$  $e_1e_2=e_3,$\newline  $e_1e_3=e_4,$  $e_2e_1=e_3,$\newline  $e_2e_2=e_4,$  $e_3e_1=e_4 $ & $\varphi(e_1)=ae_1+be_2+ce_3+de_4,$ \newline$ \varphi(e_2)=a^2e_2+2abe_3+(2ac+b^2)e_4,$\newline $\varphi(e_3)=a^3e_3+3a^2be_4,$ $\varphi(e_4)=a^4e_4,$ &$\langle e_4\rangle$\\ \hline
\end{tabular}

\end{footnotesize}

where $a, b, c, d, e, f, g, h, i, j, k, l, \alpha \in \mathbb{C}.$

\end{thm}

\section{Four-dimensional complex anti-dendriform algebras associated to the four-dimensional associative algebras with a one-dimensional center}

In this section we will classify four-dimensional anti-dendriform algebras associated to the four-dimensional associative algebras with a one-dimensional center. According to Theorem \ref{table} the algebras $As_4^{3}, As_4^{6}, As_4^{8}, As_4^{9}, As_4^{10}, As_4^{13}, As_4^{14}, As_4^{15}(\alpha)$ and  $As_4^{16}$ have a one-dimensional center. It should be noted that the associative algebra $As_4^{16}$ is a four-dimensional null-filiiform algebra and by the Theorem 3.1 in \cite{aans25} there is no compatible structure of an antidendriform algebra on $As_4^{16}$.

\begin{thm}\label{theo_3.1}
Any four-dimensional complex anti-dendriform algebra associated to the algebra $As_4^3$ is isomorphic to one of the following pairwise non-isomorphic algebras:

$AD_4^1: \ e_1\vtr e_2=e_4, \ e_3\vtr e_1=e_4;$

$AD_4^2: \ e_1\vtr e_2=e_4, \ e_2\vtr e_1=-e_4,\ e_3\vtr e_3=e_4, e_2\vtl e_1=e_4, \ e_3\vtl e_1=e_4, \ e_3\vtl e_3=-e_4;$

$AD_4^{3}: \ \begin{cases}e_1\vtr e_1=e_4, \ e_1\vtr e_2=e_4, \ e_2\vtr e_1=-e_4, \ e_3\vtr e_3=e_4,\\
e_1\vtl e_1=-e_4,  \ e_2\vtl e_1=e_4, \ e_3\vtl e_1=e_4, \ e_3\vtl e_3=-e_4;\end{cases}$

$AD_4^{4}: \ \begin{cases}e_1\vtr e_2=e_4, \ e_1\vtr e_3=e_4, \ e_2\vtr e_1=-e_4, \ e_2\vtr e_2=e_4, \ e_3\vtr e_1=e_4\\
e_1\vtl e_3=-e_4, \ e_2\vtl e_1=e_4, \ e_2\vtl e_2=-e_4;
\end{cases}$

$AD_4^{5}(\alpha): \ \begin{cases}e_1\vtr e_1=e_4,\ e_1\vtr e_2=\alpha e_4, \ e_2\vtr e_1=-\alpha e_4, \ e_2\vtr e_2=e_4, \ e_3\vtr e_3=e_4, \ \alpha\in \mathbb{C},\\
e_1\vtl e_1=-e_4,\ e_1\vtl e_2=(1-\alpha) e_4, \ e_2\vtl e_1=\alpha e_4, \\ e_2\vtl e_2=-e_4, \ e_3\vtl e_1=e_4, \ e_3\vtl e_3=-e_4;
\end{cases}$

$AD_4^{6}: \ \begin{cases}
e_1\vtr e_1= e_4,\ e_1\vtr e_2=e_3+e_4,\ e_2\vtr e_1=-e_3,\ e_3\vtr e_1=e_4,\\
e_1\vtl e_1=- e_4,\ e_1\vtl e_2=-e_3,\ e_2\vtl e_1=e_3;\end{cases}$

$AD_4^{7}(\alpha): \ \begin{cases} e_1\vtr e_2=e_3+e_4,\ e_2\vtr e_1=-e_3+\alpha e_4,\ e_3\vtr e_1=e_4,\\ e_1\vtl e_2=-e_3,\ e_2\vtl e_1=e_3-\alpha e_4, \ \alpha\in \mathbb{C};\end{cases}$

$ AD_4^{8}(\alpha): \ \begin{cases}
e_1\vtr e_1=\alpha e_4,\ e_1\vtr e_2=e_3+e_4,\ e_2\vtr e_1=-e_3,\ e_2\vtr e_2=e_4,\ e_3\vtr e_1=e_4,\\
e_1\vtl e_1=-\alpha e_4,\ e_1\vtl e_2=-e_3,\ e_2\vtl e_1=e_3,\ e_2\vtl e_2=-e_4;
\end{cases}$

$AD_4^{9}(\alpha, \beta): \ \begin{cases}
e_1\vtr e_1=e_3,\ e_1\vtr e_2=(1+\alpha)e_4,\ e_2\vtr e_1=\beta e_4,\ e_3\vtr e_1=e_4,\\
e_1\vtl e_1=-e_3,\ e_1\vtl e_2=-\alpha e_4,\ e_2\vtl e_1=-\beta e_4;
\end{cases}$

$AD_4^{10}(\alpha): \ \begin{cases}
e_1\vtr e_1=e_3,\ e_1\vtr e_2=(1+\alpha)e_4,\  e_2\vtr e_2=e_4,\ e_3\vtr e_1=e_4,\\
e_1\vtl e_1=-e_3,\ e_1\vtl e_2=-\alpha e_4,\ e_2\vtl e_2=-e_4;
\end{cases}$

$AD_4^{11}(\alpha, \delta): \ \begin{cases}
e_1\vtr e_1=\delta e_4,\ e_1\vtr e_2=e_3+e_4,\ e_2\vtr e_1=\alpha e_4,\ e_3\vtr e_1=e_4,\\
e_1\vtl e_1=-\delta e_4,\ e_1\vtl e_2=-e_3,\ e_2\vtl e_1=-\alpha e_4, \ \delta\in \mathbb{C};
\end{cases}$

$AD_4^{12}(\alpha,\beta): \ \begin{cases}
e_1\vtr e_1=\alpha e_4,\ e_1\vtr e_2=e_3+e_4,\ e_2\vtr e_1=\beta e_4,\ e_2\vtr e_2=e_4,\ e_3\vtr e_1=e_4,\\
e_1\vtl e_1=-\alpha e_4,\ e_1\vtl e_2=-e_3,\ e_2\vtl e_1=-\beta e_4,\ e_2\vtl e_2=-e_4.
\end{cases}$

$AD_4^{13}[\lambda](\alpha,\beta,\gamma):\ \begin{cases}
e_1\vtr e_1=e_3,\ e_1\vtr e_2=\lambda e_3+(1+\alpha)e_4,\ e_2\vtr e_1=\beta e_4,\\ e_2\vtr e_2=e_3+\gamma e_4,\ e_3\vtr e_1=e_4,\\
e_1\vtl e_1=-e_3,\ e_1\vtl e_2=-\lambda e_3-\alpha e_4,\ e_2\vtl e_1=-\beta e_4,\ e_2\vtl e_2=-e_3-\gamma e_4,
\end{cases}
$
where if $\lambda\neq 0$, then $\gamma \in \mathbb{C}$ and if $\lambda=0$, then $\gamma\geq0$.

\end{thm}
\begin{proof}
By considering \eqref{id5} for the following triples
$$\{e_1,e_2,e_1\},\
\{e_1,e_1,e_2\},\
\{e_1,e_2,e_2\}, \
\{e_2,e_1,e_2\},\
\{e_3,e_3,e_1\},\ \{e_3,e_1,e_3\},\
\{e_4,e_3,e_1\}.$$ we get
$e_4\vtr e_1=0, \ e_1\vtl e_4=0, \ e_4\vtr e_2=0, \ e_2\vtl e_4=0, \ e_3\vtl e_4=0, \ e_4\vtr e_3=0, \ e_4\vtl e_4=0$, respectively.

By $e_4e_i=e_4\vtr e_i+e_4\vtl e_i=0$ and $e_ie_4=e_i\vtr e_4+e_i\vtl e_4=0$ these imply $e_4\vtl e_i=0$ and $e_i\vtr e_4=0$. Then it is easy to see that $\langle e_4\rangle$ is the center of the algebras $(As_4^3, \cdot)$ and
 $(As_4^3, \vtr, \vtl)$. If we take the algebra $As_4^3/\langle e_4\rangle,$ this algebra is a three-dimensional associative Abelian algebra, and we know that any three-dimensional complex anti-dendriform algebra associated with the Abelian algebra is isomorphic to one of the pairwise non-isomorphic algebras: $AD_3^3-AD_3^7(\lambda).$

According to Proposition \ref{compatible} one can write

\textbf{Case $(As_4^3/\langle e_4\rangle, \vtr, \vtl)\cong AD_3^3$}
\[
\begin{cases}
e_1\vtr e_1=\alpha_{11}e_4,\\
e_1\vtr e_2=(1+\alpha_{12})e_4,\\
e_1\vtr e_3=\alpha_{13}e_4,\\
e_2\vtr e_1=\alpha_{21}e_4,\\
e_2\vtr e_2=\alpha_{22}e_4,\\
e_2\vtr e_3=\alpha_{23}e_4,\\
e_3\vtr e_1=(1+\alpha_{31})e_4,\\
e_3\vtr e_2=\alpha_{32}e_4,\\
e_3\vtr e_3=\alpha_{33}e_4,
\end{cases}
\qquad
\begin{cases}
e_1\vtl e_1=-\alpha_{11}e_4,\\
e_1\vtl e_2=-\alpha_{12}e_4,\\
e_1\vtl e_3=-\alpha_{13}e_4,\\
e_2\vtl e_1=-\alpha_{21}e_4,\\
e_2\vtl e_2=-\alpha_{22}e_4,\\
e_2\vtl e_3=-\alpha_{23}e_4,\\
e_3\vtl e_1=-\alpha_{31}e_4,\\
e_3\vtl e_2=-\alpha_{32}e_4,\\
e_3\vtl e_3=-\alpha_{33}e_4.
\end{cases}
\]

It is not difficult to see that $(As_4^3, \vtr, \vtl)$ are 2-nilpotent and have three generator elements. Thus we have $(x\vtr y)\vtr z=x\vtr (y\vtr z)=0$ which implies that $(As_4^3, \vtr)$ is also associative 2-nilpotent algebra and has three generator elements. According to Theorem \ref{table} there are five non-isomorphic four-dimensional indecomposable associative 2-nilpotent and the three generated algebras. Hence, we get $AD_4^{1}-AD_4^{5}(\alpha)$ algebras.

It is possible to consider the multiplication $\vtl$ as above. However, it is not difficult to show that the constructed algebras are isomorphic.

\textbf{Case $(As_4^3/\langle e_4\rangle, \vtr, \vtl)\cong AD_3^4$}

\[
\begin{cases}
e_1\vtr e_1=\alpha_{11}e_4,\\
e_1\vtr e_2=e_3+(1+\alpha_{12})e_4,\\
e_1\vtr e_3=\alpha_{13}e_4,\\
e_2\vtr e_1=-e_3+\alpha_{21}e_4,\\
e_2\vtr e_2=\alpha_{22}e_4,\\
e_2\vtr e_3=\alpha_{23}e_4,\\
e_3\vtr e_1=(1+\alpha_{31})e_4,\\
e_3\vtr e_2=\alpha_{32}e_4,\\
e_3\vtr e_3=\alpha_{33}e_4,
\end{cases}
\qquad
\begin{cases}
e_1\vtl e_1=-\alpha_{11}e_4,\\
e_1\vtl e_2=-e_3-\alpha_{12}e_4,\\
e_1\vtl e_3=-\alpha_{13}e_4,\\
e_2\vtl e_1=e_3-\alpha_{21}e_4,\\
e_2\vtl e_2=-\alpha_{22}e_4,\\
e_2\vtl e_3=-\alpha_{23}e_4,\\
e_3\vtl e_1=-\alpha_{31}e_4,\\
e_3\vtl e_2=-\alpha_{32}e_4,\\
e_3\vtl e_3=-\alpha_{33}e_4.
\end{cases}
\]

By the basis change $e'_3=e_3+\alpha_{12}e_4$, we can rewrite
\[
\begin{cases}
e_1\vtr e_1=\alpha_{11}e_4,\\
e_1\vtr e_2=e_3+e_4,\\
e_1\vtr e_3=\alpha_{13}e_4,\\
e_2\vtr e_1=-e_3+\alpha_{21}e_4,\\
e_2\vtr e_2=\alpha_{22}e_4,\\
e_2\vtr e_3=\alpha_{23}e_4,\\
e_3\vtr e_1=(1+\alpha_{31})e_4,\\
e_3\vtr e_2=\alpha_{32}e_4,\\
e_3\vtr e_3=\alpha_{33}e_4,
\end{cases}
\qquad
\begin{cases}
e_1\vtl e_1=-\alpha_{11}e_4,\\
e_1\vtl e_2=-e_3,\\
e_1\vtl e_3=-\alpha_{13}e_4,\\
e_2\vtl e_1=e_3-\alpha_{21}e_4,\\
e_2\vtl e_2=-\alpha_{22}e_4,\\
e_2\vtl e_3=-\alpha_{23}e_4,\\
e_3\vtl e_1=-\alpha_{31}e_4,\\
e_3\vtl e_2=-\alpha_{32}e_4,\\
e_3\vtl e_3=-\alpha_{33}e_4.
\end{cases}
\]
Consider \eqref{id6} for the triple $\{e_1,e_2,e_i\}$:
\[
-(e_1e_2)\vtr e_i=(e_1\vtl e_2)\vtl e_i\ \Leftrightarrow \
-e_4\vtr e_i=-e_3\vtl e_i\ \Leftrightarrow \
e_3\vtl e_i=0.
\]
This implies $\alpha_{3i}=0, \ i\in\{1,2,3\}$. Similarly, from \eqref{id3} for the triple $\{e_i,e_1,e_2\}$ one can get $\alpha_{i3}=0, \ i\in\{1,2,3\}$. No we  rewrite
\[
\begin{cases}
e_1\vtr e_1=\alpha_{11}e_4,\ e_1\vtr e_2=e_3+e_4,\ e_2\vtr e_1=-e_3+\alpha_{21}e_4,\ e_2\vtr e_2=\alpha_{22}e_4,\ e_3\vtr e_1=e_4,\\
e_1\vtl e_1=-\alpha_{11}e_4,\ e_1\vtl e_2=-e_3,\ e_2\vtl e_1=e_3-\alpha_{21}e_4,\ e_2\vtl e_2=-\alpha_{22}e_4.
\end{cases}
\]

Let us consider the general change of the generators of basis:
\[\varphi(e_1)=ae_1+ce_2-ce_3+de_4,\ \varphi(e_2)=be_2+ee_4,\ \varphi(e_3)=be_3+fe_4,\ \varphi(e_4)=abe_4,\]
where $ab\neq0$.

We express the new basis elements $\{\varphi(e_1), \varphi(e_2),\varphi(e_3), \varphi(e_4)\}$  via the basis elements $\{e_1, e_2, e_3, e_4\}.$  By verifying all the multiplications of the algebra in the new basis we obtain the relations between the parameters $\{\alpha'_{11}, \alpha'_{21}, \alpha'_{22}\}$ and $\{\alpha_{11}, \alpha_{21}, \alpha_{22}\}$:

$$
\begin{array}{lll}
\varphi(e_1)\vtr \varphi(e_2)=\varphi(e_3)+\varphi(e_4) & \Rightarrow & a=1, \  f=bc\alpha_{22},\\[1mm]
\varphi(e_1)\vtr \varphi(e_1)=\alpha^\prime_{11}\varphi(e_4) & \Rightarrow & \alpha^\prime_{11}=\frac{\alpha_{11}+c\alpha_{21}+c^2\alpha_{22}}{b},\\[1mm]
\varphi(e_2)\vtr \varphi(e_1)=-\varphi(e_3)+\alpha^\prime_{21}\varphi(e_4) & \Rightarrow & \alpha^\prime_{21}=\alpha_{21}+2c\alpha_{22},\\[1mm]
\varphi(e_2)\vtr \varphi(e_2)=\alpha^\prime_{22}\varphi(e_4) & \Rightarrow & \alpha^\prime_{22}=b\alpha_{22}.\\[1mm]
\end{array}
$$

We have the following cases:

\begin{itemize}
    \item Let $\alpha_{22}=0$. Then we have $\alpha'_{22}=0, \ \alpha'_{21}=\alpha_{21},
    \ \alpha'_{11}=\displaystyle\frac{\alpha_{11}+c\alpha_{21}}{b}$. If $\alpha_{21}=0$, we get the algebra $(0,0,0)$ and $(1,0,0)$ when $\alpha_{11}=0$ and $\alpha_{11}\neq0$ with choosing $b=\alpha_{11}$, respectively i.e. $AD_4^{6}(\delta), \ \delta\in\{0,1\}$. If $\alpha_{21}\neq0$, then choosing $c=-\frac{\alpha_{11}}{\alpha_{21}}$, we obtain the algebra $(0, \alpha,0), \ \alpha\neq0$ i.e. $AD_4^{7}(\alpha)$. So, in this case, we get the algebras  $AD_4^{6}$ and $AD_4^{7}(\alpha), \ \alpha\in \mathbb{C}.$

\item Let $\alpha_{22}\neq0$. Then choosing $b=\frac{1}{\alpha_{22}}$ and $c=-\frac{\alpha_{21}}{2\alpha_{22}}$, we obtain the algebra $AD_4^{8}(\alpha)$.
\end{itemize}

\textbf{Case $(As_4^3/\langle e_4\rangle, \vtr, \vtl)\cong AD_3^5$}

\[
\begin{cases}
e_1\vtr e_1=e_3+\alpha_{11}e_4,\\
e_1\vtr e_2=(1+\alpha_{12})e_4,\\
e_1\vtr e_3=\alpha_{13}e_4,\\
e_2\vtr e_1=\alpha_{21}e_4,\\
e_2\vtr e_2=\alpha_{22}e_4,\\
e_2\vtr e_3=\alpha_{23}e_4,\\
e_3\vtr e_1=(1+\alpha_{31})e_4,\\
e_3\vtr e_2=\alpha_{32}e_4,\\
e_3\vtr e_3=\alpha_{33}e_4,
\end{cases}
\qquad
\begin{cases}
e_1\vtl e_1=-e_3-\alpha_{11}e_4,\\
e_1\vtl e_2=-\alpha_{12}e_4,\\
e_1\vtl e_3=-\alpha_{13}e_4,\\
e_2\vtl e_1=-\alpha_{21}e_4,\\
e_2\vtl e_2=-\alpha_{22}e_4,\\
e_2\vtl e_3=-\alpha_{23}e_4,\\
e_3\vtl e_1=-\alpha_{31}e_4,\\
e_3\vtl e_2=-\alpha_{32}e_4,\\
e_3\vtl e_3=-\alpha_{33}e_4.
\end{cases}
\]

By the basis change $e'_3=e_3+\alpha_{11}e_4$, we can get
\[
\begin{cases}
e_1\vtr e_1=e_3,\\
e_1\vtr e_2=(1+\alpha_{12})e_4,\\
e_1\vtr e_3=\alpha_{13}e_4,\\
e_2\vtr e_1=\alpha_{21}e_4,\\
e_2\vtr e_2=\alpha_{22}e_4,\\
e_2\vtr e_3=\alpha_{23}e_4,\\
e_3\vtr e_1=(1+\alpha_{31})e_4,\\
e_3\vtr e_2=\alpha_{32}e_4,\\
e_3\vtr e_3=\alpha_{33}e_4,
\end{cases}
\qquad
\begin{cases}
e_1\vtl e_1=-e_3,\\
e_1\vtl e_2=-\alpha_{12}e_4,\\
e_1\vtl e_3=-\alpha_{13}e_4,\\
e_2\vtl e_1=-\alpha_{21}e_4,\\
e_2\vtl e_2=-\alpha_{22}e_4,\\
e_2\vtl e_3=-\alpha_{23}e_4,\\
e_3\vtl e_1=-\alpha_{31}e_4,\\
e_3\vtl e_2=-\alpha_{32}e_4,\\
e_3\vtl e_3=-\alpha_{33}e_4.
\end{cases}
\]
Consider \eqref{id6} for the triple $\{e_1,e_1,e_i\}$:
\[
-(e_1e_1)\vtr e_i=(e_1\vtl e_1)\vtl e_i\ \Leftrightarrow \ e_3\vtl e_i=0.
\]
This implies $\alpha_{3i}=0, \ i\in\{1,2,3\}$. Similarly, from \eqref{id3} for the triple $\{e_i,e_1,e_1\}$ one can get $\alpha_{i3}=0, \ i\in\{1,2,3\}$. No we  rewrite
\[
\begin{cases}
e_1\vtr e_1=e_3,\ e_1\vtr e_2=(1+\alpha_{12})e_4,\ e_2\vtr e_1=\alpha_{21}e_4,\ e_2\vtr e_2=\alpha_{22}e_4,\ e_3\vtr e_1=e_4,\\
e_1\vtl e_1=-e_3,\ e_1\vtl e_2=-\alpha_{12}e_4,\ e_2\vtl e_1=-\alpha_{21}e_4,\ e_2\vtl e_2=-\alpha_{22}e_4.
\end{cases}
\]

Let us consider the general change of the generators of basis:
\[\varphi(e_1)=ae_1+ce_2-ce_3+de_4,\ \varphi(e_2)=be_2+ee_4,\ \varphi(e_3)=be_3+fe_4,\ \varphi(e_4)=abe_4,\]
where $ab\neq0$.

We express the new basis elements $\{\varphi(e_1), \varphi(e_2),\varphi(e_3), \varphi(e_4)\}$  via the basis elements $\{e_1, e_2, e_3, e_4\}.$  By verifying all the multiplications of the algebra in the new basis we obtain the relations between the parameters $\{\alpha'_{12}, \alpha'_{21}, \alpha'_{22}\}$ and $\{\alpha_{12}, \alpha_{21}, \alpha_{22}\}$:

$$
\begin{array}{lll}
\varphi(e_1)\vtr \varphi(e_1)=\varphi(e_3) & \Rightarrow & b=a^2, \  f=ac\alpha_{12}+ac\alpha_{21}+c^2\alpha_{22},\\[1mm]
\varphi(e_1)\vtr \varphi(e_2)=(1+\alpha^\prime_{12})\varphi(e_4) & \Rightarrow & \alpha^\prime_{12}=\alpha_{12}+\frac{c}{a}\alpha_{22},\\[1mm]
\varphi(e_2)\vtr \varphi(e_1)=-\varphi(e_3)+\alpha^\prime_{21}\varphi(e_4) & \Rightarrow & \alpha^\prime_{21}=\alpha_{21}+\frac{c}{a}\alpha_{22},\\[1mm]
\varphi(e_2)\vtr \varphi(e_2)=\alpha^\prime_{22}\varphi(e_4) & \Rightarrow & \alpha^\prime_{22}=a\alpha_{22}.\\[1mm]
\end{array}
$$

Then we have the following cases:
\begin{itemize}
    \item Let $\alpha_{22}=0$. Then we get the algebra $AD_4^{9}(\alpha, \beta)$.

\item Let $\alpha_{22}\neq0$. Then choosing $a=\frac{1}{\alpha_{22}}$ and $c=-\frac{\alpha_{21}}{\alpha_{22}^2}$ we obtain the algebra $AD_4^{10}(\alpha)$.
\end{itemize}

\textbf{Case $(As_4^3/\langle e_4\rangle, \vtr, \vtl)\cong AD_3^6$}
\[
\begin{cases}
e_1\vtr e_1=\alpha_{11}e_4,\\
e_1\vtr e_2=e_3+(1+\alpha_{12})e_4,\\
e_1\vtr e_3=\alpha_{13}e_4,\\
e_2\vtr e_1=\alpha_{21}e_4,\\
e_2\vtr e_2=\alpha_{22}e_4,\\
e_2\vtr e_3=\alpha_{23}e_4,\\
e_3\vtr e_1=(1+\alpha_{31})e_4,\\
e_3\vtr e_2=\alpha_{32}e_4,\\
e_3\vtr e_3=\alpha_{33}e_4,
\end{cases}
\qquad
\begin{cases}
e_1\vtl e_1=-\alpha_{11}e_4,\\
e_1\vtl e_2=-e_3-\alpha_{12}e_4,\\
e_1\vtl e_3=-\alpha_{13}e_4,\\
e_2\vtl e_1=-\alpha_{21}e_4,\\
e_2\vtl e_2=-\alpha_{22}e_4,\\
e_2\vtl e_3=-\alpha_{23}e_4,\\
e_3\vtl e_1=-\alpha_{31}e_4,\\
e_3\vtl e_2=-\alpha_{32}e_4,\\
e_3\vtl e_3=-\alpha_{33}e_4.
\end{cases}
\]

By the basis change $e'_3=e_3+\alpha_{12}e_4$, we can get
\[
\begin{cases}
e_1\vtr e_1=\alpha_{11}e_4,\\
e_1\vtr e_2=e_3+e_4,\\
e_1\vtr e_3=\alpha_{13}e_4,\\
e_2\vtr e_1=\alpha_{21}e_4,\\
e_2\vtr e_2=\alpha_{22}e_4,\\
e_2\vtr e_3=\alpha_{23}e_4,\\
e_3\vtr e_1=(1+\alpha_{31})e_4,\\
e_3\vtr e_2=\alpha_{32}e_4,\\
e_3\vtr e_3=\alpha_{33}e_4,
\end{cases}
\qquad
\begin{cases}
e_1\vtl e_1=-\alpha_{11}e_4,\\
e_1\vtl e_2=-e_3,\\
e_1\vtl e_3=-\alpha_{13}e_4,\\
e_2\vtl e_1=-\alpha_{21}e_4,\\
e_2\vtl e_2=-\alpha_{22}e_4,\\
e_2\vtl e_3=-\alpha_{23}e_4,\\
e_3\vtl e_1=-\alpha_{31}e_4,\\
e_3\vtl e_2=-\alpha_{32}e_4,\\
e_3\vtl e_3=-\alpha_{33}e_4.
\end{cases}
\]
Consider \eqref{id6} for the triple $\{e_1,e_2,e_i\}$:
\[
-(e_1e_2)\vtr e_i=(e_1\vtl e_2)\vtl e_i \ \Leftrightarrow \ -e_4\vtr e_i=-e_3\vtl e_i,\ \Leftrightarrow \
e_3\vtl e_i=0.
\]
This implies $\alpha_{3i}=0, \ i\in\{1,2,3\}$. Similarly, from \eqref{id3} for the triple $\{e_i,e_1,e_2\}$ one can get $\alpha_{i3}=0, \ i\in\{1,2,3\}$. Thus, we have the following table of multiplications:
\[
\begin{cases}
e_1\vtr e_1=\alpha_{11}e_4,\ e_1\vtr e_2=e_3+e_4,\ e_2\vtr e_1=\alpha_{21}e_4,\ e_2\vtr e_2=\alpha_{22}e_4,\ e_3\vtr e_1=e_4,\\
e_1\vtl e_1=-\alpha_{11}e_4,\ e_1\vtl e_2=-e_3,\ e_2\vtl e_1=-\alpha_{21}e_4,\ e_2\vtl e_2=-\alpha_{22}e_4.
\end{cases}
\]

Let us consider the general change of the generators of basis:
\[\varphi(e_1)=ae_1+ce_2-ce_3+de_4,\ \varphi(e_2)=be_2+ee_4,\ \varphi(e_3)=be_3+fe_4,\ \varphi(e_4)=abe_4,\]
where $ab\neq0$.

We express the new basis elements $\{\varphi(e_1), \varphi(e_2),\varphi(e_3), \varphi(e_4)\}$  via the basis elements $\{e_1, e_2, e_3, e_4\}.$  By verifying all the multiplications of the algebra in the new basis we obtain the relations between the parameters $\{\alpha'_{11}, \alpha'_{21}, \alpha'_{22}\}$ and $\{\alpha_{11}, \alpha_{21}, \alpha_{22}\}$:

$$
\begin{array}{lll}
\varphi(e_1)\vtr \varphi(e_2)=\varphi(e_3)+\varphi(e_4) & \Rightarrow & a=1, \  f=bc\alpha_{22},\\[1mm]
\varphi(e_1)\vtr \varphi(e_1)=\alpha^\prime_{11}\varphi(e_4) & \Rightarrow & c=0, \ \alpha^\prime_{11}=\frac{1}{b}\alpha_{11},\\[1mm]
\varphi(e_2)\vtr \varphi(e_1)=\alpha^\prime_{21}\varphi(e_4) & \Rightarrow & \alpha^\prime_{21}=\alpha_{21},\\[1mm]
\varphi(e_2)\vtr \varphi(e_2)=\alpha^\prime_{22}\varphi(e_4) & \Rightarrow & \alpha^\prime_{22}=b\alpha_{22}.\\[1mm]
\end{array}
$$

Thus, we have the following cases:
\begin{itemize}
    \item Let $\alpha_{22}=0$. If $\alpha_{11}=0$, then we get the algebra $AD_4^{11}(\alpha, 0)$. If $\alpha_{11}\neq0$, choosing $b=\frac{1}{\alpha_{11}}$, we obtain the algebra $AD_4^{11}(\alpha, 1)$.

\item  Let $\alpha_{22}\neq0$. Then choosing $b=\frac{1}{\alpha_{22}}$ we obtain the algebra $AD_4^{12}(\alpha,\beta)$.
\end{itemize}

\textbf{Case $(As_4^3/\langle e_4\rangle, \vtr, \vtl)\cong AD_3^7(\lambda)$}
\[
\begin{cases}
e_1\vtr e_1=e_3+\alpha_{11}e_4,\\
e_1\vtr e_2=\lambda e_3+(1+\alpha_{12})e_4,\\
e_1\vtr e_3=\alpha_{13}e_4,\\
e_2\vtr e_1=\alpha_{21}e_4,\\
e_2\vtr e_2=e_3+\alpha_{22}e_4,\\
e_2\vtr e_3=\alpha_{23}e_4,\\
e_3\vtr e_1=(1+\alpha_{31})e_4,\\
e_3\vtr e_2=\alpha_{32}e_4,\\
e_3\vtr e_3=\alpha_{33}e_4,
\end{cases}
\qquad
\begin{cases}
e_1\vtl e_1=-e_3-\alpha_{11}e_4,\\
e_1\vtl e_2=-\lambda e_3-\alpha_{12}e_4,\\
e_1\vtl e_3=-\alpha_{13}e_4,\\
e_2\vtl e_1=-\alpha_{21}e_4,\\
e_2\vtl e_2=-e_3-\alpha_{22}e_4,\\
e_2\vtl e_3=-\alpha_{23}e_4,\\
e_3\vtl e_1=-\alpha_{31}e_4,\\
e_3\vtl e_2=-\alpha_{32}e_4,\\
e_3\vtl e_3=-\alpha_{33}e_4.
\end{cases}
\]
where $\lambda\in \mathbb{C}.$

By the basis change $e'_3=e_3+\alpha_{11}e_4$, we can get
\[
\begin{cases}
e_1\vtr e_1=e_3,\\
e_1\vtr e_2=\lambda e_3+(1+\alpha_{12})e_4,\\
e_1\vtr e_3=\alpha_{13}e_4,\\
e_2\vtr e_1=\alpha_{21}e_4,\\
e_2\vtr e_2=e_3+\alpha_{22}e_4,\\
e_2\vtr e_3=\alpha_{23}e_4,\\
e_3\vtr e_1=(1+\alpha_{31})e_4,\\
e_3\vtr e_2=\alpha_{32}e_4,\\
e_3\vtr e_3=\alpha_{33}e_4,
\end{cases}
\qquad
\begin{cases}
e_1\vtl e_1=-e_3,\\
e_1\vtl e_2=-\lambda e_3-\alpha_{12}e_4,\\
e_1\vtl e_3=-\alpha_{13}e_4,\\
e_2\vtl e_1=-\alpha_{21}e_4,\\
e_2\vtl e_2=-e_3-\alpha_{22}e_4,\\
e_2\vtl e_3=-\alpha_{23}e_4,\\
e_3\vtl e_1=-\alpha_{31}e_4,\\
e_3\vtl e_2=-\alpha_{32}e_4,\\
e_3\vtl e_3=-\alpha_{33}e_4.
\end{cases}
\]
Consider \eqref{id6} for the triple $\{e_1,e_1,e_i\}$:
\[
-(e_1e_1)\vtr e_i=(e_1\vtl e_1)\vtl e_i\ \Leftrightarrow \
e_3\vtl e_i=0.
\]
This implies $\alpha_{3i}=0, \ i\in\{1,2,3\}$. Similarly, from \eqref{id3} for the triple $\{e_i,e_1,e_1\}$ one can get $\alpha_{i3}=0, \ i\in\{1,2,3\}$. No we  rewrite
\[
\begin{cases}
e_1\vtr e_1=e_3,\ e_1\vtr e_2=\lambda e_3+(1+\alpha_{12})e_4,\ e_2\vtr e_1=\alpha_{21}e_4,\ e_2\vtr e_2=e_3+\alpha_{22}e_4,\ e_3\vtr e_1=e_4,\\
e_1\vtl e_1=-e_3,\ e_1\vtl e_2=-\lambda e_3-\alpha_{12}e_4,\ e_2\vtl e_1=-\alpha_{21}e_4,\ e_2\vtl e_2=-e_3-\alpha_{22}e_4.
\end{cases}
\]

Let us consider the general change of the generators of basis:
\[\varphi(e_1)=ae_1+ce_2-ce_3+de_4,\ \varphi(e_2)=be_2+ee_4,\ \varphi(e_3)=be_3+fe_4,\ \varphi(e_4)=abe_4,\]
where $ab\neq0$.

We express the new basis elements $\{\varphi(e_1), \varphi(e_2),\varphi(e_3), \varphi(e_4)\}$  via the basis elements $\{e_1, e_2, e_3, e_4\}.$  By verifying all the multiplications of the algebra in the new basis we obtain the relations between the parameters $\{\alpha'_{12}, \alpha'_{21}, \alpha'_{22}\}$ and $\{\alpha_{12}, \alpha_{21}, \alpha_{22}\}$:

$$
\begin{array}{lll}
\varphi(e_1)\vtr \varphi(e_1)=\varphi(e_3) & \Rightarrow & b=a^2+ac\lambda+c^2, \  f=ac\alpha_{12}+ac\alpha_{21}+c^2\alpha_{22},\\[1mm]
\varphi(e_1)\vtr \varphi(e_2)=\lambda\varphi(e_3)+(1+\alpha^\prime_{12})\varphi(e_4) & \Rightarrow & a\lambda+c=\lambda, \ ab\alpha_{12}+bc\alpha_{22}=f\lambda+ab\alpha^\prime_{12},\\[1mm]
\varphi(e_2)\vtr \varphi(e_1)=\alpha^\prime_{21}\varphi(e_4) & \Rightarrow & c=0, \ \alpha^\prime_{21}=\alpha_{21},\\[1mm]
\varphi(e_2)\vtr \varphi(e_2)=\varphi(e_3)+\alpha^\prime_{22}\varphi(e_4) & \Rightarrow & b=1, \ \alpha^\prime_{22}=a\alpha_{22}.\\[1mm]
\end{array}
$$

From this, we get
\[
a^2=b=1,\ c=f=0, \ \alpha^\prime_{12}=\alpha_{12}, \ \alpha^\prime_{21}=\alpha_{21}, \ \alpha^\prime_{22}=a\alpha_{22}.
\]

Hence, we have $AD_4^{13}[\lambda](\alpha,\beta,\gamma)$ for $\lambda\neq 0$  and $AD_4^{13}[0](\alpha,\beta,\gamma)$, where $\gamma\geq 0$.

\end{proof}

\begin{thm}\label{theo_3.2}
Any four-dimensional complex anti-dendriform algebra associated to the algebra $As_4^6$ is isomorphic to one of the following pairwise non-isomorphic algebras:

$AD_4^{14}: \ e_1\vtr e_2=e_4, \ e_3\vtr e_1=e_4, \ e_2\vtl e_1=-e_4, \ e_3\vtl e_1=-e_4, \ e_3\vtl e_3=e_4;$

$AD_4^{15}: \ e_1\vtr e_2=e_4, \ e_2\vtr e_1=-e_4,\ e_3\vtr e_3=e_4;$

$AD_4^{16}: \ e_1\vtr e_1=e_4, \ e_1\vtr e_2=e_4, \ e_2\vtr e_1=-e_4, \ e_3\vtr e_3=e_4,\ e_1\vtl e_1=-e_4;
$

$AD_4^{17}: \ \begin{cases}e_1\vtr e_2=e_4, \ e_1\vtr e_3=e_4, \ e_2\vtr e_1=-e_4, \ e_2\vtr e_2=e_4, \ e_3\vtr e_1=e_4,\\
e_1\vtl e_3=-e_4,  \ e_2\vtl e_2=-e_4, \ e_3\vtl e_1=-e_4, \ e_3\vtl e_3=e_4;
\end{cases}$

$AD_4^{18}(\alpha): \ \begin{cases}e_1\vtr e_1=e_4,\ e_1\vtr e_2=\alpha e_4, \ e_2\vtr e_1=-\alpha e_4, \ e_2\vtr e_2=e_4, \ e_3\vtr e_3=e_4, \ \alpha\in \mathbb{C},\\
e_1\vtl e_1=-e_4,\ e_1\vtl e_2=(1-\alpha) e_4, \ e_2\vtl e_1=(\alpha-1) e_4, \ e_2\vtl e_2=-e_4.
\end{cases}
$

\end{thm}
\begin{proof}
By considering \eqref{id5} for the following triples
$$\{e_1,e_1,e_2\},\
\{e_1,e_2,e_1\},\
\{e_3,e_3,e_2\}, \
\{e_2,e_3,e_3\},\
\{e_1,e_2,e_3\},\ \{e_3,e_1,e_2\},\
\{e_3,e_3,e_4\}.$$ we get
$e_1\vtl e_4=0, \ e_4\vtr e_1=0, \ e_4\vtr e_2=0, \ e_2\vtl e_4=0, \ e_4\vtr e_3=0, \ e_3\vtl e_4=0, \ e_4\vtr e_4=0$, respectively.

By $e_4e_i=e_4\vtr e_i+e_4\vtl e_i=0$ and $e_ie_4=e_i\vtr e_4+e_i\vtl e_4=0$ these imply $e_4\vtl e_i=0$ and $e_i\vtr e_4=0$. Then it is easy to see that $\langle e_4\rangle$ is the center of the algebras $(As_4^6, \cdot)$ and
 $(As_4^6, \vtr, \vtl)$. If we take the algebra $As_4^6/\langle e_4\rangle,$ this algebra is a three-dimensional associative Abelian algebra, and we know that any three-dimensional complex anti-dendriform algebra associated with the Abelian algebra is isomorphic to one of the pairwise non-isomorphic algebras: $AD_3^3-AD_3^7(\lambda).$

According to Proposition \ref{compatible} one can write

\textbf{Case $(As_4^6/\langle e_4\rangle, \vtr, \vtl)\cong AD_3^3$}
\[
\begin{cases}
e_1\vtr e_1=\alpha_{11}e_4,\\
e_1\vtr e_2=(1+\alpha_{12})e_4,\\
e_1\vtr e_3=\alpha_{13}e_4,\\
e_2\vtr e_1=(-1+\alpha_{21})e_4,\\
e_2\vtr e_2=\alpha_{22}e_4,\\
e_2\vtr e_3=\alpha_{23}e_4,\\
e_3\vtr e_1=\alpha_{31}e_4,\\
e_3\vtr e_2=\alpha_{32}e_4,\\
e_3\vtr e_3=(1+\alpha_{33})e_4,
\end{cases}
\qquad
\begin{cases}
e_1\vtl e_1=-\alpha_{11}e_4,\\
e_1\vtl e_2=-\alpha_{12}e_4,\\
e_1\vtl e_3=-\alpha_{13}e_4,\\
e_2\vtl e_1=-\alpha_{21}e_4,\\
e_2\vtl e_2=-\alpha_{22}e_4,\\
e_2\vtl e_3=-\alpha_{23}e_4,\\
e_3\vtl e_1=-\alpha_{31}e_4,\\
e_3\vtl e_2=-\alpha_{32}e_4,\\
e_3\vtl e_3=-\alpha_{33}e_4.
\end{cases}
\]

Similarly, it is not difficult to see that $(As_4^6, \vtr, \vtl)$ are 2-nilpotent and have three generator elements. Thus we have $(x\vtr y)\vtr z=x\vtr (y\vtr z)=0$ which implies that $(As_4^6, \vtr)$ is also associative 2-nilpotent algebra and has three generator elements. According to Theorem \ref{table} there are 3 non-isomorphic four-dimensional indecomposable associative 2-nilpotent and the three generated algebras. Hence, we derive the algebras $AD_4^{14}-AD_4^{18}(\alpha)$.

It is possible to consider the multiplication $\vtl$ as above. However, it is not difficult to show that the constructed algebras are isomorphic.

\textbf{Case $(As_4^6/\langle e_4\rangle, \vtr, \vtl)\cong AD_3^4$}
\[
\begin{cases}
e_1\vtr e_1=\alpha_{11}e_4,\\
e_1\vtr e_2=e_3+(1+\alpha_{12})e_4,\\
e_1\vtr e_3=\alpha_{13}e_4,\\
e_2\vtr e_1=-e_3+(-1+\alpha_{21})e_4,\\
e_2\vtr e_2=\alpha_{22}e_4,\\
e_2\vtr e_3=\alpha_{23}e_4,\\
e_3\vtr e_1=\alpha_{31}e_4,\\
e_3\vtr e_2=\alpha_{32}e_4,\\
e_3\vtr e_3=(1+\alpha_{33})e_4,
\end{cases}
\qquad
\begin{cases}
e_1\vtl e_1=-\alpha_{11}e_4,\\
e_1\vtl e_2=-e_3-\alpha_{12}e_4,\\
e_1\vtl e_3=-\alpha_{13}e_4,\\
e_2\vtl e_1=e_3-\alpha_{21}e_4,\\
e_2\vtl e_2=-\alpha_{22}e_4,\\
e_2\vtl e_3=-\alpha_{23}e_4,\\
e_3\vtl e_1=-\alpha_{31}e_4,\\
e_3\vtl e_2=-\alpha_{32}e_4,\\
e_3\vtl e_3=-\alpha_{33}e_4.
\end{cases}
\]

Considering the equation \eqref{id4} for the triples $\{e_2,e_1,e_3\}$ and  $\{e_3,e_1,e_2\}$, we obtain the following
restrictions:

$$
e_2\vtr(e_1\vtr e_3)\stackrel{(\ref{id4})}{=}(e_2\vtl e_1)\vtl e_3 \ \Leftrightarrow \ \alpha_{13}(e_2\vtr e_4)=(e_3-\alpha_{21}e_4)\vtl e_3\  \Leftrightarrow\
\alpha_{33}=0,$$
$$e_3\vtr(e_1\vtr e_2)\stackrel{(\ref{id4})}{=}(e_3\vtl e_1)\vtl e_2 \ \Leftrightarrow\  e_3\vtr (e_3+(1+\alpha_{12})e_4)=-\alpha_{31}e_4\vtl e_2\  \Leftrightarrow \
\alpha_{33}=-1.$$
This implies a contradiction. Hence, there is not such a case.

\textbf{Case $(As_4^6/\langle e_4\rangle, \vtr, \vtl)\cong AD_3^5$}
\[
\begin{cases}
e_1\vtr e_1=e_3+\alpha_{11}e_4,\\
e_1\vtr e_2=(1+\alpha_{12})e_4,\\
e_1\vtr e_3=\alpha_{13}e_4,\\
e_2\vtr e_1=(-1+\alpha_{21})e_4,\\
e_2\vtr e_2=\alpha_{22}e_4,\\
e_2\vtr e_3=\alpha_{23}e_4,\\
e_3\vtr e_1=\alpha_{31}e_4,\\
e_3\vtr e_2=\alpha_{32}e_4,\\
e_3\vtr e_3=(1+\alpha_{33})e_4,
\end{cases}
\qquad
\begin{cases}
e_1\vtl e_1=-e_3-\alpha_{11}e_4,\\
e_1\vtl e_2=-\alpha_{12}e_4,\\
e_1\vtl e_3=-\alpha_{13}e_4,\\
e_2\vtl e_1=-\alpha_{21}e_4,\\
e_2\vtl e_2=-\alpha_{22}e_4,\\
e_2\vtl e_3=-\alpha_{23}e_4,\\
e_3\vtl e_1=-\alpha_{31}e_4,\\
e_3\vtl e_2=-\alpha_{32}e_4,\\
e_3\vtl e_3=-\alpha_{33}e_4.
\end{cases}
\]

Considering the identity \eqref{id4} for the triples $\{e_1,e_1,e_3\}$ and  $\{e_3,e_1,e_1\}$, we obtain the following
restrictions on structure constants:
$$\alpha_{33}=0, \ \ \alpha_{33}=-1.$$
This implies a contradiction. Hence, it means that there is no anti-dendriform algebra associated to the algebra $As_4^6$ and satisfying the condition $(As_4^6/\langle e_4\rangle, \vtr, \vtl)\cong AD_3^5$.

\textbf{Case $(As_4^6/\langle e_4\rangle, \vtr, \vtl)\cong AD_3^6$}
\[
\begin{cases}
e_1\vtr e_1=\alpha_{11}e_4,\\
e_1\vtr e_2=e_3+(1+\alpha_{12})e_4,\\
e_1\vtr e_3=\alpha_{13}e_4,\\
e_2\vtr e_1=(-1+\alpha_{21})e_4,\\
e_2\vtr e_2=\alpha_{22}e_4,\\
e_2\vtr e_3=\alpha_{23}e_4,\\
e_3\vtr e_1=\alpha_{31}e_4,\\
e_3\vtr e_2=\alpha_{32}e_4,\\
e_3\vtr e_3=(1+\alpha_{33})e_4,
\end{cases}
\qquad
\begin{cases}
e_1\vtl e_1=-\alpha_{11}e_4,\\
e_1\vtl e_2=-e_3-\alpha_{12}e_4,\\
e_1\vtl e_3=-\alpha_{13}e_4,\\
e_2\vtl e_1=-\alpha_{21}e_4,\\
e_2\vtl e_2=-\alpha_{22}e_4,\\
e_2\vtl e_3=-\alpha_{23}e_4,\\
e_3\vtl e_1=-\alpha_{31}e_4,\\
e_3\vtl e_2=-\alpha_{32}e_4,\\
e_3\vtl e_3=-\alpha_{33}e_4.
\end{cases}
\]

By considering the identity \eqref{id4} for the triples $\{e_1,e_2,e_3\}$ and  $\{e_3,e_1,e_2\}$, we obtain the restrictions on structure constants as follows:

$$e_1\vtr(e_2\vtr e_3)=(e_1\vtl e_2)\vtl e_3 \ \Leftrightarrow\ \alpha_{23}(e_1\vtr e_4)=(-e_3-\alpha_{12}e_4)\vtl e_3\ \Leftrightarrow\ \alpha_{33}=0,$$
$$e_3\vtr(e_1\vtr e_2)=(e_3\vtl e_1)\vtl e_2\ \Leftrightarrow\ e_3\vtr (e_3+(1+\alpha_{12})e_4)=-\alpha_{31}e_4\vtl e_2\ \Leftrightarrow\
\alpha_{33}=-1.$$

This implies a contradiction. Hence, there is not such a case.

\textbf{Case $(As_4^6/\langle e_4\rangle, \vtr, \vtl)\cong AD_3^7(\lambda)$}
\[
\begin{cases}
e_1\vtr e_1=e_3+\alpha_{11}e_4,\\
e_1\vtr e_2=\lambda e_3+(1+\alpha_{12})e_4,\\
e_1\vtr e_3=\alpha_{13}e_4,\\
e_2\vtr e_1=(-1+\alpha_{21})e_4,\\
e_2\vtr e_2=e_3+\alpha_{22}e_4,\\
e_2\vtr e_3=\alpha_{23}e_4,\\
e_3\vtr e_1=\alpha_{31}e_4,\\
e_3\vtr e_2=\alpha_{32}e_4,\\
e_3\vtr e_3=(1+\alpha_{33})e_4,
\end{cases}
\qquad
\begin{cases}
e_1\vtl e_1=-e_3-\alpha_{11}e_4,\\
e_1\vtl e_2=-\lambda e_3-\alpha_{12}e_4,\\
e_1\vtl e_3=-\alpha_{13}e_4,\\
e_2\vtl e_1=-\alpha_{21}e_4,\\
e_2\vtl e_2=-e_3-\alpha_{22}e_4,\\
e_2\vtl e_3=-\alpha_{23}e_4,\\
e_3\vtl e_1=-\alpha_{31}e_4,\\
e_3\vtl e_2=-\alpha_{32}e_4,\\
e_3\vtl e_3=-\alpha_{33}e_4.
\end{cases}
\]

Considering the identity \eqref{id4} for the triples $\{e_1,e_1,e_3\}$ and  $\{e_3,e_1,e_1\}$, we obtain the following restrictions, respectively:

\[\alpha_{33}=0, \ \ \alpha_{33}=-1.\]
This implies a contradiction. Hence, there is not a compatible anti-dendriform algebra structure on $As_4^6$ and satisfying the condition $(As_4^6/\langle e_4\rangle, \vtr, \vtl)\cong AD_3^7(\lambda)$. \end{proof}

\begin{thm}\label{theo_3.3} Any four-dimensional complex anti-dendriform algebra associated to the algebra $As_4^8$ is isomorphic to one of the following pairwise non-isomorphic algebras:

$AD_4^{19}(\alpha): \
\begin{cases}
e_1\vtr e_1=\frac12 e_3,\ e_1\vtr e_3=2e_4,\ e_2\vtr e_2=(-1+\alpha)e_4,\ e_3\vtr e_1=-e_4,\\
e_1\vtl e_1=\frac12e_3,\ e_1\vtl e_3=-e_4,\ e_2\vtl e_2=-\alpha e_4,\ e_3\vtl e_1=2e_4;
\end{cases}$

$AD_4^{20}(\alpha): \
\begin{cases}
e_1\vtr e_1=\frac12 e_3,\ e_1\vtr e_3=2e_4,\ e_2\vtr e_1=e_4,\ e_2\vtr e_2=(-1+\alpha)e_4,\ e_3\vtr e_1=-e_4,\\
e_1\vtl e_1=\frac12e_3,\ e_1\vtl e_3=-e_4,\ e_2\vtl e_1=-e_4,\ e_2\vtl e_2=-\alpha e_4,\ e_3\vtl e_1=2e_4, \ \alpha\neq-1;
\end{cases}$

$AD_4^{21}(\alpha,\beta): \
\begin{cases}
e_1\vtr e_1=\frac12 e_3+e_4,\ e_1\vtr e_3=2e_4,\ e_2\vtr e_1=\alpha e_4,\\ e_2\vtr e_2=(-1+\beta)e_4,\ e_3\vtr e_1=-e_4,\\
e_1\vtl e_1=\frac12e_3-e_4,\  e_1\vtl e_3=-e_4,\ e_2\vtl e_1=-\alpha e_4,\\ e_2\vtl e_2=-\beta e_4,\ e_3\vtl e_1=2e_4, \ \beta\neq-1;
\end{cases}$

$AD_4^{22}: \
\begin{cases}
e_1\vtr e_1=\frac12 e_3,\ e_1\vtr e_2=e_4,\ e_1\vtr e_3=2e_4,\ e_2\vtr e_2=-2e_4,\ e_3\vtr e_1=-e_4,\\
e_1\vtl e_1=\frac12e_3,\ e_1\vtl e_2=-e_4,\ e_1\vtl e_3=-e_4,\  e_2\vtl e_2=e_4,\ e_3\vtl e_1=2e_4;
\end{cases}$

$AD_4^{23}(\alpha): \
\begin{cases}
e_1\vtr e_1=\frac12 e_3+e_4,\ e_1\vtr e_2=\alpha e_4,\ e_1\vtr e_3=2e_4,\  e_2\vtr e_2=-2e_4,\ e_3\vtr e_1=-e_4,\\
e_1\vtl e_1=\frac12e_3-e_4,\ e_1\vtl e_2=-\alpha e_4,\ e_1\vtl e_3=-e_4,\  e_2\vtl e_2=e_4,\ e_3\vtl e_1=2e_4.
\end{cases}$
\end{thm}
\begin{proof}
By considering \eqref{id5} for the following triples
$$\{e_1,e_2,e_2\},\
\{e_2,e_2,e_1\},\
\{e_1,e_3,e_2\}, \
\{e_2,e_1,e_3\},\
\{e_2,e_2,e_3\},\ \{e_3,e_2,e_2\},\
\{e_2,e_2,e_4\}.$$ we get
$e_1\vtl e_4=0, \ e_4\vtr e_1=0, \ e_4\vtr e_2=0, \ e_2\vtl e_4=0, \ e_4\vtr e_3=0, \ e_3\vtl e_4=0, \ e_4\vtr e_4=0$, respectively.

By $e_4e_i=e_4\vtr e_i+e_4\vtl e_i=0$ and $e_ie_4=e_i\vtr e_4+e_i\vtl e_4=0$ these imply $e_4\vtl e_i=0$ and $e_i\vtr e_4=0$. Then it is easy to see that $\langle e_4\rangle$ is the center of the algebras $(As_4^8, \cdot)$ and
 $(As_4^8, \vtr, \vtl)$. If we take the algebra $As_4^8/\langle e_4\rangle,$ this algebra is a three-dimensional associative algebra $As^3_3$ with multiplication $e_1e_1=e_3$, and we know that any three-dimensional complex anti-dendriform algebra associated with the algebra $As_3^3$ is isomorphic to one of the pairwise non-isomorphic algebras: $AD_3^{18}(\alpha)-AD_3^{23}.$

According to Proposition \ref{compatible} one can write

\textbf{Case $(As_4^8, \vtr, \vtl)/\langle e_4\rangle\cong AD_3^{18}(\alpha)$}

\[
\begin{cases}
e_1\vtr e_1=\alpha e_3+\alpha_{11}e_4,\\
e_1\vtr e_2=\alpha_{12}e_4,\\
e_1\vtr e_3=(1+\alpha_{13})e_4,\\
e_2\vtr e_1=\alpha_{21}e_4,\\
e_2\vtr e_2=(-1+\alpha_{22})e_4,\\
e_2\vtr e_3=\alpha_{23}e_4,\\
e_3\vtr e_1=(1+\alpha_{31})e_4,\\
e_3\vtr e_2=\alpha_{32}e_4,\\
e_3\vtr e_3=\alpha_{33}e_4,
\end{cases}
\qquad
\begin{cases}
e_1\vtl e_1=(1-\alpha)e_3-\alpha_{11}e_4,\\
e_1\vtl e_2=-\alpha_{12}e_4,\\
e_1\vtl e_3=-\alpha_{13}e_4,\\
e_2\vtl e_1=-\alpha_{21}e_4,\\
e_2\vtl e_2=-\alpha_{22}e_4,\\
e_2\vtl e_3=-\alpha_{23}e_4,\\
e_3\vtl e_1=-\alpha_{31}e_4,\\
e_3\vtl e_2=-\alpha_{32}e_4,\\
e_3\vtl e_3=-\alpha_{33}e_4.
\end{cases}
\]

Considering \eqref{id2} for the triples $\{e_1,e_1,e_2\}, \ \{e_1,e_1,e_3\}$ we obtain $\alpha_{32}=0, \alpha_{33}=0$.

Considering \eqref{id2} and \eqref{id3} for the triple $\{e_2,e_1,e_1\}$ we obtain $\alpha_{23}=0$.

Considering \eqref{id1}-\eqref{id4} for the triple $\{e_1,e_1,e_1\}\}$ we obtain the following system:
\[
\begin{cases}
(1+\alpha
_{13})(-1+\alpha)-\alpha
_{31} \alpha=0,\\
(1+\alpha_{31})+(1+\alpha_{13})\alpha=0, \\
-\alpha_{13}+(1+\alpha
_{13}) \alpha=0,\\
-\alpha_{31}(-1+\alpha)+(1+\alpha
_{13}) \alpha=0.
\end{cases}
\]
Then the system has only one solution i.e. $\alpha_{13}=1, \alpha_{31}=-2, \alpha=\frac12$.

Rewrite the multiplication table
\[
\begin{cases}
e_1\vtr e_1=\frac12 e_3+\alpha_{11}e_4,\\
e_1\vtr e_2=\alpha_{12}e_4,\\
e_1\vtr e_3=2e_4,\\
e_2\vtr e_1=\alpha_{21}e_4,\\
e_2\vtr e_2=(-1+\alpha_{22})e_4,\\
e_3\vtr e_1=-e_4,
\end{cases}
\qquad
\begin{cases}
e_1\vtl e_1=\frac12e_3-\alpha_{11}e_4,\\
e_1\vtl e_2=-\alpha_{12}e_4,\\
e_1\vtl e_3=-e_4,\\
e_2\vtl e_1=-\alpha_{21}e_4,\\
e_2\vtl e_2=-\alpha_{22}e_4,\\
e_3\vtl e_1=2e_4.
\end{cases}
\]

Let us consider the general change of the generators of basis:
\[\varphi(e_1)=\sqrt[3]{a^2}e_1+be_2+ce_3+de_4,\ \varphi(e_2)=ae_2+b\sqrt[3]{a}e_3+ee_4,\ \varphi(e_3)=\sqrt[3]{a^4}e_3+(2c\sqrt[3]{a^2}-b^2)e_4,\ \varphi(e_4)=a^2e_4,\]
where $a\neq0$.

We express the new basis elements $\{\varphi(e_1), \varphi(e_2),\varphi(e_3), \varphi(e_4)\}$  via the basis elements $\{e_1, e_2, e_3, e_4\}.$  By verifying all the multiplications of the algebra in the new basis we obtain the relations between the parameters $\{\alpha'_{11}, \alpha'_{12},\alpha'_{21}, \alpha'_{22}\}$ and $\{\alpha_{11}, \alpha_{12}, \alpha_{21}, \alpha_{22}\}$:

$$
\begin{array}{lll}
\varphi(e_1)\vtr \varphi(e_1)=\frac{1}{2}\varphi(e_3)+\alpha^\prime_{11}\varphi(e_4) & \Rightarrow & a^2\alpha^\prime_{11}=\sqrt[3]{a^4}\alpha_{11}+b\sqrt[3]{a^2}(\alpha_{12}+\alpha_{21})+b^2(\alpha_{22}-\frac{1}{2}),\\[1mm]
\varphi(e_1)\vtr \varphi(e_2)=\alpha^\prime_{12}\varphi(e_4) & \Rightarrow & a\alpha^\prime_{12}=\sqrt[3]{a^2}\alpha_{12}+b(1+\alpha_{22}),\\[1mm]
\varphi(e_2)\vtr \varphi(e_1)=\alpha^\prime_{21}\varphi(e_4) & \Rightarrow & a\alpha^\prime_{21}=\sqrt[3]{a^2}\alpha_{21}+b(-2+\alpha_{22}),\\[1mm]
\varphi(e_2)\vtr \varphi(e_2)=(-1+\alpha^\prime_{22})\varphi(e_4) & \Rightarrow & \alpha^\prime_{22}=\alpha_{22}.\\[1mm]
\end{array}
$$

Thus, we have the following cases:
\begin{itemize}
    \item[(a)] Let $\alpha_{22}=2$. Then by choosing $b$ we obtain following restrictions:
 $$a\alpha^\prime_{11}=\sqrt[3]{a}\alpha_{11},\ \alpha^\prime_{12}=0,\ a\alpha^\prime_{21}=\sqrt[3]{a^2}\alpha_{21},\ \alpha^\prime_{22}=2.
$$
\begin{itemize}
  \item Let $\alpha_{11}=0.$ If $\alpha_{21}=0$, then we get the algebras $AD_4^{19}(2)$. If $\alpha_{21}\neq0$, then by putting $a=\sqrt[3]{\alpha_{21}}$ we obtain the algebra $AD_4^{20}(2).$
  \item Let $\alpha_{11}\neq0.$ Then applying $a=\sqrt[3]{\alpha_{11}^2}$ we obtain the algebra $AD_4^{21}(\alpha,2).$
\end{itemize}

\item[(b)] Let $\alpha_{22}=-1$. Then by choosing $b$ we obtain following restrictions:
 $$a\alpha^\prime_{11}=\sqrt[3]{a}\alpha_{11},\ a\alpha^\prime_{12}=\sqrt[3]{a^2}\alpha_{12},\ \alpha^\prime_{21}=0,\ \alpha^\prime_{22}=-1.
$$
\begin{itemize}
  \item Let $\alpha_{11}=0.$ If $\alpha_{12}=0$, then we get the algebras $AD_4^{19}(-1)$. If $\alpha_{12}\neq0$, then by putting $a=\sqrt[3]{\alpha_{12}}$ we obtain the algebra $AD_4^{22}.$
  \item Let $\alpha_{11}\neq0.$ Then applying $a=\sqrt[3]{\alpha_{11}^2}$ we obtain the algebra $AD_4^{23}(\alpha).$
\end{itemize}

\item[(c)] Let $\alpha_{22}\notin\{-1,2\}$. Then by choosing $b$ we obtain following restrictions:
 $$a\alpha^\prime_{11}=\sqrt[3]{a}\alpha_{11},\ \alpha^\prime_{12}=0,\ a\alpha^\prime_{21}=\sqrt[3]{a^2}\alpha_{12},\ \alpha^\prime_{22}=\alpha_{22}.
$$
\begin{itemize}
  \item Let $\alpha_{11}=0.$ If $\alpha_{21}=0$, then we get the algebras $AD_4^{19}(\alpha)$, where $\alpha\notin\{-1,2\}$. If $\alpha_{21}\neq0$, then by putting $a=\sqrt[3]{\alpha_{21}}$ we obtain the algebra $AD_4^{20}(\alpha),$ where $\alpha\notin\{-1,2\}$.
  \item Let $\alpha_{11}\neq0.$ Then applying $a=\sqrt[3]{\alpha_{11}^2}$ we obtain the algebra $AD_4^{21}(\alpha, \beta),$ where $\beta\notin\{-1,2\}$.
\end{itemize}

\end{itemize}

\textbf{Case $(As_4^8, \vtr, \vtl)/\langle e_4\rangle\cong AD_3^{19}$}
\[
\begin{cases}
e_1\vtr e_1=\alpha_{11}e_4,\\
e_1\vtr e_2=\alpha_{12}e_4,\\
e_1\vtr e_3=(1+\alpha_{13})e_4,\\
e_2\vtr e_1=e_3+\alpha_{21}e_4,\\
e_2\vtr e_2=(-1+\alpha_{22})e_4,\\
e_2\vtr e_3=\alpha_{23}e_4,\\
e_3\vtr e_1=(1+\alpha_{31})e_4,\\
e_3\vtr e_2=\alpha_{32}e_4,\\
e_3\vtr e_3=\alpha_{33}e_4,
\end{cases}
\qquad
\begin{cases}
e_1\vtl e_1=e_3-\alpha_{11}e_4,\\
e_1\vtl e_2=-\alpha_{12}e_4,\\
e_1\vtl e_3=-\alpha_{13}e_4,\\
e_2\vtl e_1=-e_3-\alpha_{21}e_4,\\
e_2\vtl e_2=-\alpha_{22}e_4,\\
e_2\vtl e_3=-\alpha_{23}e_4,\\
e_3\vtl e_1=-\alpha_{31}e_4,\\
e_3\vtl e_2=-\alpha_{32}e_4,\\
e_3\vtl e_3=-\alpha_{33}e_4.
\end{cases}
\]

By considering \eqref{id3} for the triples $\{e_1,e_1,e_1\}$ and $\{e_1,e_2,e_1\}$ we get
$$\alpha_{13}=0,\ \alpha_{13}=-1.$$
This implies a contradiction. Hence, it means that there is no anti-dendriform algebra associated to
the algebra $As_4^8$ and satisfying the condition $(As_4^8, \vtr, \vtl)/\langle e_4\rangle\cong AD_3^{19}$.

\textbf{Case $(As_4^8, \vtr, \vtl)/\langle e_4\rangle\cong AD_3^{20}(\alpha)$}
\[
\begin{cases}
e_1\vtr e_1=\alpha e_3+\alpha_{11}e_4,\\
e_1\vtr e_2=e_3+\alpha_{12}e_4,\\
e_1\vtr e_3=(1+\alpha_{13})e_4,\\
e_2\vtr e_1=-e_3+\alpha_{21}e_4,\\
e_2\vtr e_2=(-1+\alpha_{22})e_4,\\
e_2\vtr e_3=\alpha_{23}e_4,\\
e_3\vtr e_1=(1+\alpha_{31})e_4,\\
e_3\vtr e_2=\alpha_{32}e_4,\\
e_3\vtr e_3=\alpha_{33}e_4,
\end{cases}
\qquad
\begin{cases}
e_1\vtl e_1=(1-\alpha)e_3-\alpha_{11}e_4,\\
e_1\vtl e_2=-e_3-\alpha_{12}e_4,\\
e_1\vtl e_3=-\alpha_{13}e_4,\\
e_2\vtl e_1=e_3-\alpha_{21}e_4,\\
e_2\vtl e_2=-\alpha_{22}e_4,\\
e_2\vtl e_3=-\alpha_{23}e_4,\\
e_3\vtl e_1=-\alpha_{31}e_4,\\
e_3\vtl e_2=-\alpha_{32}e_4,\\
e_3\vtl e_3=-\alpha_{33}e_4.
\end{cases}
\]

Considering the equation \eqref{id3} for the triples $\{e_1,e_1,e_1\}$ and  $\{e_1,e_1,e_2\}$, we obtain the following
restrictions:
$$\alpha(1+\alpha_{13})=-\alpha_{13},\ \alpha_{13}=-1$$
This implies a contradiction. Hence, there is not such a case.

\textbf{Case $(As_4^8, \vtr, \vtl)/\langle e_4\rangle\cong AD_3^{21}(\alpha)$}
\[
\begin{cases}
e_1\vtr e_1=\alpha_{11}e_4,\\
e_1\vtr e_2=e_3+\alpha_{12}e_4,\\
e_1\vtr e_3=(1+\alpha_{13})e_4,\\
e_2\vtr e_1=\alpha e_3+\alpha_{21}e_4,\\
e_2\vtr e_2=(-1+\alpha_{22})e_4,\\
e_2\vtr e_3=\alpha_{23}e_4,\\
e_3\vtr e_1=(1+\alpha_{31})e_4,\\
e_3\vtr e_2=\alpha_{32}e_4,\\
e_3\vtr e_3=\alpha_{33}e_4,
\end{cases}
\qquad
\begin{cases}
e_1\vtl e_1=e_3-\alpha_{11}e_4,\\
e_1\vtl e_2=-e_3-\alpha_{12}e_4,\\
e_1\vtl e_3=-\alpha_{13}e_4,\\
e_2\vtl e_1=-\alpha e_3-\alpha_{21}e_4,\\
e_2\vtl e_2=-\alpha_{22}e_4,\\
e_2\vtl e_3=-\alpha_{23}e_4,\\
e_3\vtl e_1=-\alpha_{31}e_4,\\
e_3\vtl e_2=-\alpha_{32}e_4,\\
e_3\vtl e_3=-\alpha_{33}e_4.
\end{cases}
\]
where $\alpha\neq-1$.

By considering the identity \eqref{id3} for the triples  $\{e_1,e_1,e_1\}$  and  $\{e_1,e_1,e_2\}$, we obtain the restrictions on structure constants as follows:
$$e_1\vtr(e_1\vtr e_1)=-e_1\vtl (e_1\cdot e_1)\ \Leftrightarrow\ \alpha_{11}e_1\vtr e_4=e_1\vtl e_3\ \Leftrightarrow\ \alpha_{13}=0,$$
$$e_1\vtr(e_1\vtr e_2)=-e_1\vtl (e_1\cdot e_2)\ \Leftrightarrow\ e_1\vtr(e_3+\alpha_{12}e_4)=0 \ \Leftrightarrow\ \alpha_{13}=-1.$$
This derives a contradiction. Hence, there is not such a case.

\textbf{Case $(As_4^8, \vtr, \vtl)/\langle e_4\rangle\cong AD_3^{22}(\alpha, \beta)$}
\[
\begin{cases}
e_1\vtr e_1=\alpha e_3+\alpha_{11}e_4,\\
e_1\vtr e_2=\alpha_{12}e_4,\\
e_1\vtr e_3=(1+\alpha_{13})e_4,\\
e_2\vtr e_1=\beta e_3+\alpha_{21}e_4,\\
e_2\vtr e_2=e_3+(-1+\alpha_{22})e_4,\\
e_2\vtr e_3=\alpha_{23}e_4,\\
e_3\vtr e_1=(1+\alpha_{31})e_4,\\
e_3\vtr e_2=\alpha_{32}e_4,\\
e_3\vtr e_3=\alpha_{33}e_4,
\end{cases}
\qquad
\begin{cases}
e_1\vtl e_1=(1-\alpha)e_3-\alpha_{11}e_4,\\
e_1\vtl e_2=-\alpha_{12}e_4,\\
e_1\vtl e_3=-\alpha_{13}e_4,\\
e_2\vtl e_1=-\beta e_3-\alpha_{21}e_4,\\
e_2\vtl e_2=-e_3-\alpha_{22}e_4,\\
e_2\vtl e_3=-\alpha_{23}e_4,\\
e_3\vtl e_1=-\alpha_{31}e_4,\\
e_3\vtl e_2=-\alpha_{32}e_4,\\
e_3\vtl e_3=-\alpha_{33}e_4.
\end{cases}
\]

Considering the identity \eqref{id3} for the triples $\{e_1,e_1,e_1\}$ and  $\{e_1,e_2,e_2\}$, we obtain the following
restrictions on structure constants:
$$\alpha(1+\alpha_{13})=-\alpha_{13}, \ \ \alpha_{13}=-1.$$
This implies a contradiction. Hence, it means that there is no anti-dendriform algebra associated to the algebra $As_4^8$ and satisfying the condition $(As_4^8, \vtr, \vtl)/\langle e_4\rangle\cong AD_3^{22}(\alpha, \beta)$.

\textbf{Case $(As_4^8, \vtr, \vtl)/\langle e_4\rangle\cong AD_3^{23}$}
\[
\begin{cases}
e_1\vtr e_1=e_2+\alpha_{11}e_4,\\
e_1\vtr e_2=\alpha_{12}e_4,\\
e_1\vtr e_3=(1+\alpha_{13})e_4,\\
e_2\vtr e_1=\alpha_{21}e_4,\\
e_2\vtr e_2=(-1+\alpha_{22})e_4,\\
e_2\vtr e_3=\alpha_{23}e_4,\\
e_3\vtr e_1=(1+\alpha_{31})e_4,\\
e_3\vtr e_2=\alpha_{32}e_4,\\
e_3\vtr e_3=\alpha_{33}e_4,
\end{cases}
\qquad
\begin{cases}
e_1\vtl e_1=-e_2+e_3-\alpha_{11}e_4,\\
e_1\vtl e_2=-\alpha_{12}e_4,\\
e_1\vtl e_3=-\alpha_{13}e_4,\\
e_2\vtl e_1=-\alpha_{21}e_4,\\
e_2\vtl e_2=-\alpha_{22}e_4,\\
e_2\vtl e_3=-\alpha_{23}e_4,\\
e_3\vtl e_1=-\alpha_{31}e_4,\\
e_3\vtl e_2=-\alpha_{32}e_4,\\
e_3\vtl e_3=-\alpha_{33}e_4.
\end{cases}
\]

Considering the some equations for the triples $\{e_1,e_1,e_2\}, \{e_2,e_1,e_1\}$ and $\{e_1,e_1,e_2\}$ we obtain the following
restrictions:

$$
e_1\vtr(e_1\vtr e_2)\stackrel{(\ref{id2})}{=}-(e_1\cdot e_1)\vtr e_2 \ \Leftrightarrow \ \alpha_{12}e_1\vtr e_4=-e_3\vtr e_2\  \Leftrightarrow\
\alpha_{32}=0,$$
$$e_2\vtr(e_1\vtr e_1)\stackrel{(\ref{id2})}{=}-(e_2\cdot e_1)\vtr e_1 \ \Leftrightarrow\  e_2\vtr (e_2+\alpha_{11}e_4)=0\  \Leftrightarrow \
\alpha_{22}=1,$$
$$e_1\vtr(e_1\vtr e_2)\stackrel{(\ref{id4})}{=}-(e_1\vtl e_1)\vtl e_2 \ \Leftrightarrow\  \alpha_{12}e_1\vtr e_4=(-e_2+e_3-\alpha_{11}e_4)\vtl e_2\  \Leftrightarrow \
\alpha_{22}=\alpha_{32}.$$
This implies a contradiction. Hence, there is not such a case.

\end{proof}

\begin{thm}\label{theo_3.4}
Any four-dimensional complex anti-dendriform algebra associated to the algebra $As_4^9$ is isomorphic to one of the following pairwise non-isomorphic algebras:

$AD_4^{24}(\alpha,\beta):\  \begin{cases}
e_1\vtr e_1=\frac12 e_3,\ e_1\vtr e_2=\alpha e_4,\ e_1\vtr e_3=-2e_4,\ e_2\vtr e_1=(1+\beta)e_4,\ e_3\vtr e_1=e_4,\\
e_1\vtl e_1=\frac12e_3,\ e_1\vtl e_2=-\alpha e_4,\ e_1\vtl e_3=e_4,\ e_2\vtl e_1=-\beta e_4,\ e_3\vtl e_1=-2e_4;
\end{cases}$

$AD_4^{25}(\alpha,\beta):\  \begin{cases}
e_1\vtr e_1=\frac12 e_3,\ e_1\vtr e_2=\alpha e_4,\ e_1\vtr e_3=-2e_4,\\ e_2\vtr e_1=(1+\beta)e_4,\ e_2\vtr e_2=e_4,\ e_3\vtr e_1=e_4,\\
e_1\vtl e_1=\frac12e_3,\ e_1\vtl e_2=-\alpha e_4,\ e_1\vtl e_3=e_4,\\ e_2\vtl e_1=-\beta e_4,\ e_2\vtl e_2=-e_4,\ e_3\vtl e_1=-2e_4;
\end{cases}$

$AD_4^{26}(\alpha):\ \begin{cases}
e_1\vtr e_1=e_2,\ e_1\vtr e_2=\alpha e_4,\ e_1\vtr e_3=(-1+\alpha )e_4,\ e_2\vtr e_1=2e_4,\ e_3\vtr e_1=-\alpha e_4,\\
e_1\vtl e_1=-e_2+e_3,\ e_1\vtl e_2=-\alpha e_4,\ e_1\vtl e_3=-\alpha e_4,\\ e_2\vtl e_1=-e_4,\ e_3\vtl e_1=(-1+\alpha )e_4.
\end{cases}
$

\end{thm}
\begin{proof}
By considering \eqref{id5} for the following triples
$$\{e_1,e_2,e_1\},\
\{e_1,e_3,e_1\},\
\{e_2,e_1,e_2\}, \
\{e_2,e_3,e_1\},\
\{e_4,e_3,e_1\},\ \{e_1,e_1,e_4\},\
\{e_4,e_1,e_1\}.$$ we get
$e_1\vtl e_4=0, \ e_4\vtr e_1=0, \ e_4\vtr e_2=0, \ e_2\vtl e_4=0, \ e_4\vtl e_4=0, \ e_3\vtr e_4=0, \ e_4\vtl e_3=0$, respectively.

By $e_4e_i=e_4\vtr e_i+e_4\vtl e_i=0$ and $e_ie_4=e_i\vtr e_4+e_i\vtl e_4=0$ these imply $e_4\vtl e_i=0$ and $e_i\vtr e_4=0$. Then it is easy to see that $\langle e_4\rangle$ is the center of the algebras $(As_4^9, \cdot)$ and $(As_4^9, \vtr, \vtl)$. If we take the algebra $As_4^9/\langle e_4\rangle,$ this algebra is a three-dimensional associative algebra $As^3_3$ with multiplication $e_1e_1=e_3$, and we know that any three-dimensional complex anti-dendriform algebra associated with the algebra $As_3^3$ is isomorphic to one of the pairwise non-isomorphic algebras $AD_3^{18}(\alpha)-AD_3^{23}.$

According to Proposition \ref{compatible} one can write

\textbf{Case $(As_4^9, \vtr, \vtl)/\langle e_4\rangle\cong AD_3^{18}(\alpha)$}
\[
\begin{cases}
e_1\vtr e_1=\alpha e_3+\alpha_{11}e_4,\\
e_1\vtr e_2=\alpha_{12}e_4,\\
e_1\vtr e_3=(-1+\alpha_{13})e_4,\\
e_2\vtr e_1=(1+\alpha_{21})e_4,\\
e_2\vtr e_2=\alpha_{22}e_4,\\
e_2\vtr e_3=\alpha_{23}e_4,\\
e_3\vtr e_1=(-1+\alpha_{31})e_4,\\
e_3\vtr e_2=\alpha_{32}e_4,\\
e_3\vtr e_3=\alpha_{33}e_4,
\end{cases}
\qquad
\begin{cases}
e_1\vtl e_1=(1-\alpha)e_3-\alpha_{11}e_4,\\
e_1\vtl e_2=-\alpha_{12}e_4,\\
e_1\vtl e_3=-\alpha_{13}e_4,\\
e_2\vtl e_1=-\alpha_{21}e_4,\\
e_2\vtl e_2=-\alpha_{22}e_4,\\
e_2\vtl e_3=-\alpha_{23}e_4,\\
e_3\vtl e_1=-\alpha_{31}e_4,\\
e_3\vtl e_2=-\alpha_{32}e_4,\\
e_3\vtl e_3=-\alpha_{33}e_4.
\end{cases}
\]

Considering \eqref{id2} and \eqref{id3} for the triple $\{e_1,e_1,e_1\}$, we get the following system of equations:
\[
\begin{cases}
-1+\alpha_{31}+(-1+\alpha_{13})\alpha=0,\\
-\alpha_{13}+(-1+\alpha_{13})\alpha=0.
\end{cases}
\]
From this system, we obtain that $\alpha_{31}=1-\alpha_{13}$.

Similarly, if we consider \eqref{id4} for the triple $\{e_1,e_1,e_1\}$, we get the relation $(\alpha_{13}-1)(2\alpha-1)=0$. So we have the following system of equations:
\[
\begin{cases}
(\alpha_{13}-1)(2\alpha-1)=0,\\
\alpha_{13}(\alpha-1)-\alpha=0.
\end{cases}
\]

If we assume $\alpha\neq\frac12$, then from the first equation of the system we get $\alpha_{13}=1$, and if we put this value into the second equation of the system, we get the contradiction $-1=0$. So $\alpha=\frac12$. From the above system, we find $\alpha_{13}=-1$ and $\alpha_{31}=2$.

By considering \eqref{id2} for the triples $\{e_1,e_1,e_2\}, \ \{e_1,e_1,e_3\}$ and $\{e_2,e_1,e_1\}$, we derive $$\alpha_{23}=\alpha_{32}=\alpha_{33}=0.$$

So, the multiplication table of the algebra has the following form:
\[
\begin{cases}
e_1\vtr e_1=\frac12 e_3+\alpha_{11}e_4,\\
e_1\vtr e_2=\alpha_{12}e_4,\\
e_1\vtr e_3=-2e_4,\\
e_2\vtr e_1=(1+\alpha_{21})e_4,\\
e_2\vtr e_2=\alpha_{22}e_4,\\
e_3\vtr e_1=e_4,
\end{cases}
\qquad
\begin{cases}
e_1\vtl e_1=\frac12e_3-\alpha_{11}e_4,\\
e_1\vtl e_2=-\alpha_{12}e_4,\\
e_1\vtl e_3=e_4,\\
e_2\vtl e_1=-\alpha_{21}e_4,\\
e_2\vtl e_2=-\alpha_{22}e_4,\\
e_3\vtl e_1=-2e_4.
\end{cases}
\]

Let us consider the general change of the generators of basis:
\[\varphi(e_1)=ae_1+be_2+ce_3+ee_4,\ \varphi(e_2)=a^2e_2+de_4,\ \varphi(e_3)=a^2e_3+a(b-2c)e_4,\ \varphi(e_4)=a^3e_4,\]
where $a\neq0$.

We express the new basis elements $\{\varphi(e_1), \varphi(e_2),\varphi(e_3), \varphi(e_4)\}$  via the basis elements $\{e_1, e_2, e_3, e_4\}.$  By verifying all the multiplications of the algebra in the new basis we obtain the relations between the parameters $\{\alpha'_{11}, \alpha'_{12},\alpha'_{21}, \alpha'_{22}\}$ and $\{\alpha_{11}, \alpha_{12}, \alpha_{21}, \alpha_{22}\}$:
$$
\begin{array}{lll}
\varphi(e_1)\vtr \varphi(e_1)=\frac{1}{2}\varphi(e_3)+\alpha^\prime_{11}\varphi(e_4) & \Rightarrow & a^3\alpha^\prime_{11}=a^2\alpha_{11}+ab(\frac12+\alpha_{12}+\alpha_{21})+b^2\alpha_{22},\\[1mm]
\varphi(e_1)\vtr \varphi(e_2)=\alpha^\prime_{12}\varphi(e_4) & \Rightarrow & a\alpha^\prime_{12}=a\alpha_{12}+b\alpha_{22},\\[1mm]
\varphi(e_2)\vtr \varphi(e_1)=(1+\alpha^\prime_{21})\varphi(e_4) & \Rightarrow & a\alpha^\prime_{21}=a\alpha_{21}+b\alpha_{22},\\[1mm]
\varphi(e_2)\vtr \varphi(e_2)=\alpha^\prime_{22}\varphi(e_4) & \Rightarrow & \alpha^\prime_{22}=a\alpha_{22}.\\[1mm]
\end{array}
$$

We have the following cases:

\begin{itemize}
    \item[(a)] Let $\alpha_{22}=0.$ Then if $\alpha_{11}\neq0$, with  choosing $a=-\frac{b(1+2\alpha_{12}+2\alpha_{21})}{2\alpha_{11}}$ and if  $\alpha_{11}=0$ with choosing $b=0$, we can conclude $\alpha^\prime_{11}=0$. So, we get the algebra $AD_4^{24}(\alpha,\beta)$.

    \item[(b)] Let $\alpha_{22}\neq0$. Choosing $a=\frac{1}{\alpha_{22}}$ and the suitable value of $b$ we derive $\alpha^\prime_{11}=0$ and $\alpha^\prime_{22}=0$, obtain the algebra $AD_4^{25}(\alpha,\beta)$.
\end{itemize}

Now if $(As_4^9, \vtr, \vtl)/\langle e_4\rangle$ ante-dendriform algebra is isomorphic to one of the algebras $AD_3^{19}$, $AD_3^{20}(\alpha),$ $AD_3^{21 }(\alpha)$ and $AD_3^{22}(\alpha, \beta)$, then we will proof there is not a four-dimensional complex anti-dendriform algebra associated to the associative algebra $As_4^9.$

\textbf{Case $(As_4^9, \vtr, \vtl)/\langle e_4\rangle\cong AD_3^{19}$}
\[
\begin{cases}
e_1\vtr e_1=\alpha_{11}e_4,\\
e_1\vtr e_2=\alpha_{12}e_4,\\
e_1\vtr e_3=(-1+\alpha_{13})e_4,\\
e_2\vtr e_1=e_3+(1+\alpha_{21})e_4,\\
e_2\vtr e_2=\alpha_{22}e_4,\\
e_2\vtr e_3=\alpha_{23}e_4,\\
e_3\vtr e_1=(-1+\alpha_{31})e_4,\\
e_3\vtr e_2=\alpha_{32}e_4,\\
e_3\vtr e_3=\alpha_{33}e_4,
\end{cases}
\qquad
\begin{cases}
e_1\vtl e_1=e_3-\alpha_{11}e_4,\\
e_1\vtl e_2=-\alpha_{12}e_4,\\
e_1\vtl e_3=-\alpha_{13}e_4,\\
e_2\vtl e_1=-e_3-\alpha_{21}e_4,\\
e_2\vtl e_2=-\alpha_{22}e_4,\\
e_2\vtl e_3=-\alpha_{23}e_4,\\
e_3\vtl e_1=-\alpha_{31}e_4,\\
e_3\vtl e_2=-\alpha_{32}e_4,\\
e_3\vtl e_3=-\alpha_{33}e_4.
\end{cases}
\]

By consediring the identity \eqref{id3} for the triples $\{e_1,e_1,e_1\}$ and $\{e_1,e_2,e_1\}$ we get the contradiction $\alpha_{13}=0$ and $\alpha_{13}=1$. Therefore, there is not such a case.

\textbf{Case $(As_4^9, \vtr, \vtl)/\langle e_4\rangle\cong AD_3^{20}(\alpha)$}
\[
\begin{cases}
e_1\vtr e_1=\alpha e_3+\alpha_{11}e_4,\\
e_1\vtr e_2=e_3+\alpha_{12}e_4,\\
e_1\vtr e_3=(-1+\alpha_{13})e_4,\\
e_2\vtr e_1=-e_3+(1+\alpha_{21})e_4,\\
e_2\vtr e_2=\alpha_{22}e_4,\\
e_2\vtr e_3=\alpha_{23}e_4,\\
e_3\vtr e_1=(-1+\alpha_{31})e_4,\\
e_3\vtr e_2=\alpha_{32}e_4,\\
e_3\vtr e_3=\alpha_{33}e_4,
\end{cases}
\qquad
\begin{cases}
e_1\vtl e_1=(1-\alpha)e_3-\alpha_{11}e_4,\\
e_1\vtl e_2=-e_3-\alpha_{12}e_4,\\
e_1\vtl e_3=-\alpha_{13}e_4,\\
e_2\vtl e_1=e_3-\alpha_{21}e_4,\\
e_2\vtl e_2=-\alpha_{22}e_4,\\
e_2\vtl e_3=-\alpha_{23}e_4,\\
e_3\vtl e_1=-\alpha_{31}e_4,\\
e_3\vtl e_2=-\alpha_{32}e_4,\\
e_3\vtl e_3=-\alpha_{33}e_4.
\end{cases}
\]

Consider \eqref{id3} for the triples $\{e_1,e_1,e_2\}$ and $\{e_1,e_1,e_1\}$ we obtain $\alpha_{13}=1$ and $\alpha_{13}=0$ which derives a contradiction. Hence, there is not such a case.

\textbf{Case $(As_4^9, \vtr, \vtl)/\langle e_4\rangle\cong AD_3^{21}(\alpha)$}
\[
\begin{cases}
e_1\vtr e_1=\alpha_{11}e_4,\\
e_1\vtr e_2=e_3+\alpha_{12}e_4,\\
e_1\vtr e_3=(-1+\alpha_{13})e_4,\\
e_2\vtr e_1=\alpha e_3+(1+\alpha_{21})e_4,\\
e_2\vtr e_2=\alpha_{22}e_4,\\
e_2\vtr e_3=\alpha_{23}e_4,\\
e_3\vtr e_1=(-1+\alpha_{31})e_4,\\
e_3\vtr e_2=\alpha_{32}e_4,\\
e_3\vtr e_3=\alpha_{33}e_4,
\end{cases}
\qquad
\begin{cases}
e_1\vtl e_1=e_3-\alpha_{11}e_4,\\
e_1\vtl e_2=-e_3-\alpha_{12}e_4,\\
e_1\vtl e_3=-\alpha_{13}e_4,\\
e_2\vtl e_1=-\alpha e_3-\alpha_{21}e_4,\\
e_2\vtl e_2=-\alpha_{22}e_4,\\
e_2\vtl e_3=-\alpha_{23}e_4,\\
e_3\vtl e_1=-\alpha_{31}e_4,\\
e_3\vtl e_2=-\alpha_{32}e_4,\\
e_3\vtl e_3=-\alpha_{33}e_4.
\end{cases}
\]
where $\alpha\neq-1$.

Considering the identity \eqref{id3} for the triples $\{e_1,e_1,e_1\}$ and  $\{e_1,e_1,e_2\}$, we obtain the following restrictions, respectively:
\[\alpha_{13}=0, \ \ \alpha_{13}=1.\]
This implies a contradiction. Hence, there is not a compatible anti-dendriform algebra structure on $As_4^9.$

\textbf{Case $(As_4^9, \vtr, \vtl)/\langle e_4\rangle\cong AD_3^{22}(\alpha,\beta)$}
\[
\begin{cases}
e_1\vtr e_1=\alpha e_3+\alpha_{11}e_4,\\
e_1\vtr e_2=\alpha_{12}e_4,\\
e_1\vtr e_3=(-1+\alpha_{13})e_4,\\
e_2\vtr e_1=\beta e_3+(1+\alpha_{21})e_4,\\
e_2\vtr e_2=e_3+\alpha_{22}e_4,\\
e_2\vtr e_3=\alpha_{23}e_4,\\
e_3\vtr e_1=(-1+\alpha_{31})e_4,\\
e_3\vtr e_2=\alpha_{32}e_4,\\
e_3\vtr e_3=\alpha_{33}e_4,
\end{cases}
\qquad
\begin{cases}
e_1\vtl e_1=(1-\alpha)e_3-\alpha_{11}e_4,\\
e_1\vtl e_2=-\alpha_{12}e_4,\\
e_1\vtl e_3=-\alpha_{13}e_4,\\
e_2\vtl e_1=-\beta e_3-\alpha_{21}e_4,\\
e_2\vtl e_2=-e_3-\alpha_{22}e_4,\\
e_2\vtl e_3=-\alpha_{23}e_4,\\
e_3\vtl e_1=-\alpha_{31}e_4,\\
e_3\vtl e_2=-\alpha_{32}e_4,\\
e_3\vtl e_3=-\alpha_{33}e_4.
\end{cases}
\]

Similarly, consider \eqref{id3} for the triples $\{e_1,e_2,e_2\}$ and $\{e_1,e_1,e_1\}$ we can conclude $\alpha_{13}=1$ and $\alpha_{13}=0$
which derives a contradiction. Hence, there is not such a case.

\textbf{Case $(As_4^9, \vtr, \vtl)/\langle e_4\rangle\cong AD_3^{23}$}
\[
\begin{cases}
e_1\vtr e_1=e_2+\alpha_{11}e_4,\\
e_1\vtr e_2=\alpha_{12}e_4,\\
e_1\vtr e_3=(-1+\alpha_{13})e_4,\\
e_2\vtr e_1=(1+\alpha_{21})e_4,\\
e_2\vtr e_2=\alpha_{22}e_4,\\
e_2\vtr e_3=\alpha_{23}e_4,\\
e_3\vtr e_1=(-1+\alpha_{31})e_4,\\
e_3\vtr e_2=\alpha_{32}e_4,\\
e_3\vtr e_3=\alpha_{33}e_4,
\end{cases}
\qquad
\begin{cases}
e_1\vtl e_1=-e_2+e_3-\alpha_{11}e_4,\\
e_1\vtl e_2=-\alpha_{12}e_4,\\
e_1\vtl e_3=-\alpha_{13}e_4,\\
e_2\vtl e_1=-\alpha_{21}e_4,\\
e_2\vtl e_2=-\alpha_{22}e_4,\\
e_2\vtl e_3=-\alpha_{23}e_4,\\
e_3\vtl e_1=-\alpha_{31}e_4,\\
e_3\vtl e_2=-\alpha_{32}e_4,\\
e_3\vtl e_3=-\alpha_{33}e_4.
\end{cases}
\]

Considering the some identity, we obtain the following restrictions on structure
constants:

$$\begin{tabular}{p{2.5cm}|p{2.5cm}|p{2.5cm}}
    \hline
  Triple & Identity  & restrictions \\
  \hline
    $\{e_1,e_1,e_2\}$   & $\eqref{id2}$& $\alpha_{32}=0$, \\[1mm]
  \hline
    $\{e_1,e_1,e_3\}$   & $\eqref{id2}$& $\alpha_{33}=0$, \\[1mm]
  \hline
    $\{e_2,e_1,e_1\}$   & $\eqref{id2}$& $\alpha_{22}=0$, \\[1mm]
  \hline
    $\{e_1,e_1,e_1\}$   & $\eqref{id3}$& $\alpha_{13}=\alpha_{12}$, \\[1mm]
  \hline
  $\{e_2,e_1,e_1\}$   & $\eqref{id3}$& $\alpha_{23}=0$, \\[1mm]
  \hline
   $\{e_1,e_1,e_1\}$   & $\eqref{id1}$& $\alpha_{21}=1$, \\[1mm]
  \hline
   $\{e_1,e_1,e_1\}$   & $\eqref{id4}$& $\alpha_{31}=1-\alpha_{12}$. \\[1mm]
  \hline
   \end{tabular}$$

By the basis change $e'_2=e_2+\alpha_{11}e_4$ we can put $\alpha_{11}=0.$ Thus, the table of multiplications of the algebra has the form:
\[\begin{cases}
e_1\vtr e_1=e_2,\ e_1\vtr e_2=\alpha_{12}e_4,\ e_1\vtr e_3=(-1+\alpha_{12})e_4,\ e_2\vtr e_1=2e_4,\ e_3\vtr e_1=-\alpha_{12}e_4,\\
e_1\vtl e_1=-e_2+e_3,\ e_1\vtl e_2=-\alpha_{12}e_4,\ e_1\vtl e_3=-\alpha_{12}e_4,\ e_2\vtl e_1=-e_4,\ e_3\vtl e_1=(-1+\alpha_{12})e_4.
\end{cases}\]

By using the general change in the algebra we can show that $\alpha'_{12}=\alpha_{12}.$
Hence, we get the algebra $AD_4^{26}(\alpha)$.
\end{proof}

\begin{thm}\label{theo_3.5}
Any four-dimensional complex anti-dendriform algebra associated to the algebra $As_4^{10}$ is isomorphic to one of the following pairwise non-isomorphic algebras:

$AD_4^{27}: \ e_1\vtr e_2=e_4, \ e_3\vtr e_1=e_4,\ e_1\vtl e_1=e_4, \ e_2\vtl e_1=-e_4, \ e_3\vtl e_1=-e_4, \ e_3\vtl e_3=e_4;
$

$AD_4^{28}: \ e_1\vtr e_2=e_4, \ e_3\vtr e_3=e_4,\ e_2\vtr e_1=-e_4, \ e_1\vtl e_1=e_4;
$

$AD_4^{29}: \ e_1\vtr e_1=e_4, \  e_1\vtr e_2=e_4, \  e_2\vtr e_1=-e_4, \  e_3\vtr e_3=e_4;$

$AD_4^{30}: \ \begin{cases}e_1\vtr e_2=e_4, \ e_1\vtr e_3=e_4, \ e_2\vtr e_1=-e_4, \ e_2\vtr e_2=e_4, \ e_3\vtr e_1=e_4,\\
e_1\vtl e_1=e_4,  \ e_1\vtl e_3=-e_4,  \ e_2\vtl e_2=-e_4, \ e_3\vtl e_1=-e_4, \ e_3\vtl e_3=e_4;
\end{cases}$

$AD_4^{31}(\alpha): \ \begin{cases}e_1\vtr e_1=e_4,\ e_1\vtr e_2=\alpha e_4, \ e_2\vtr e_1=-\alpha e_4, \ e_2\vtr e_2=e_4, \ e_3\vtr e_3=e_4,\\
e_1\vtl e_2=(1-\alpha) e_4, \ e_2\vtl e_1=(\alpha-1) e_4, \ e_2\vtl e_2=-e_4.
\end{cases}
$

\end{thm}
\begin{proof}
By considering \eqref{id5} for the following triples
$$\{e_1,e_3,e_3\},\
\{e_3,e_3,e_1\},\
\{e_3,e_3,e_2\}, \
\{e_2,e_3,e_3\},\
\{e_1,e_1,e_3\},\ \{e_3,e_1,e_1\},\
\{e_3,e_3,e_4\}.$$ we get
$e_1\vtl e_4=0, \ e_4\vtr e_1=0, \ e_4\vtr e_2=0, \ e_2\vtl e_4=0, \ e_4\vtr e_3=0, \ e_3\vtl e_4=0, \ e_4\vtr e_4=0$, respectively.

By $e_4e_i=e_4\vtr e_i+e_4\vtl e_i=0$ and $e_ie_4=e_i\vtr e_4+e_i\vtl e_4=0$ these imply $e_4\vtl e_i=0$ and $e_i\vtr e_4=0$. Then it is easy to see that $\langle e_4\rangle$ is the center of the algebras $(As_4^{10}, \cdot)$ and $(As_4^{10}, \vtr, \vtl)$. If we take the algebra $As_4^{10}/\langle e_4\rangle,$ this algebra is a three-dimensional associative Abelian algebra, and we know that any three-dimensional complex anti-dendriform algebra associated with the Abelian algebra is isomorphic to one of the pairwise non-isomorphic algebras $AD_3^3-AD_3^7(\lambda).$

According to Proposition \ref{compatible} one can write

\textbf{Case $(As_4^{10}, \vtr, \vtl)/\langle e_4\rangle\cong AD_3^{3}$}
\[
\begin{cases}
e_1\vtr e_1=(1+\alpha_{11})e_4,\\
e_1\vtr e_2=(1+\alpha_{12})e_4,\\
e_1\vtr e_3=\alpha_{13}e_4,\\
e_2\vtr e_1=(-1+\alpha_{21})e_4,\\
e_2\vtr e_2=\alpha_{22}e_4,\\
e_2\vtr e_3=\alpha_{23}e_4,\\
e_3\vtr e_1=\alpha_{31}e_4,\\
e_3\vtr e_2=\alpha_{32}e_4,\\
e_3\vtr e_3=(1+\alpha_{33})e_4,
\end{cases}
\qquad
\begin{cases}
e_1\vtl e_1=-\alpha_{11}e_4,\\
e_1\vtl e_2=-\alpha_{12}e_4,\\
e_1\vtl e_3=-\alpha_{13}e_4,\\
e_2\vtl e_1=-\alpha_{21}e_4,\\
e_2\vtl e_2=-\alpha_{22}e_4,\\
e_2\vtl e_3=-\alpha_{23}e_4,\\
e_3\vtl e_1=-\alpha_{31}e_4,\\
e_3\vtl e_2=-\alpha_{32}e_4,\\
e_3\vtl e_3=-\alpha_{33}e_4.
\end{cases}
\]

It is not difficult to see that $(As_4^{10}, \vtr, \vtl)$ are 2-nilpotent and have three generator elements. Thus we have $(x\vtr y)\vtr z=x\vtr (y\vtr z)=0$ which implies that $(As_4^{10}, \vtr)$ is also associative 2-nilpotent algebra and has three generator elements. According to Theorem \ref{table} there are five non-isomorphic four-dimensional indecomposable associative 2-nilpotent and the three generated algebras. Hence, we get $AD_4^{27}-AD_4^{31}(\alpha)$ algebras.

It is possible to consider the multiplication $\vtl$ as above. However, it is not difficult to show that the constructed algebras are isomorphic.

Now if $(As_4^{10}, \vtr, \vtl)/\langle e_4\rangle$ ante-dendriform algebra is isomorphic to one of the algebras $AD_3^{4}-AD_3^{7}(\alpha)$, then we will proof there is not a four-dimensional complex anti-dendriform algebra associated to the associative algebra $As_4^{10}.$

\textbf{Case $(As_4^{10}, \vtr, \vtl)/\langle e_4\rangle\cong AD_3^{4}$}
\[
\begin{cases}
e_1\vtr e_1=(1+\alpha_{11})e_4,\\
e_1\vtr e_2=e_3+(1+\alpha_{12})e_4,\\
e_1\vtr e_3=\alpha_{13}e_4,\\
e_2\vtr e_1=-e_3+(-1+\alpha_{21})e_4,\\
e_2\vtr e_2=\alpha_{22}e_4,\\
e_2\vtr e_3=\alpha_{23}e_4,\\
e_3\vtr e_1=\alpha_{31}e_4,\\
e_3\vtr e_2=\alpha_{32}e_4,\\
e_3\vtr e_3=(1+\alpha_{33})e_4,
\end{cases}
\qquad
\begin{cases}
e_1\vtl e_1=-\alpha_{11}e_4,\\
e_1\vtl e_2=-e_3-\alpha_{12}e_4,\\
e_1\vtl e_3=-\alpha_{13}e_4,\\
e_2\vtl e_1=e_3-\alpha_{21}e_4,\\
e_2\vtl e_2=-\alpha_{22}e_4,\\
e_2\vtl e_3=-\alpha_{23}e_4,\\
e_3\vtl e_1=-\alpha_{31}e_4,\\
e_3\vtl e_2=-\alpha_{32}e_4,\\
e_3\vtl e_3=-\alpha_{33}e_4.
\end{cases}
\]

Consider \eqref{id4} for the triples $\{e_1,e_2,e_3\}$
 and $\{e_3,e_1,e_2\}$ we obtain $\alpha_{33}=0$ and $\alpha_{33}=-1$ which derives a contradiction. Hence, there is not such a case.

\textbf{Case $(As_4^{10}, \vtr, \vtl)/\langle e_4\rangle\cong AD_3^{5}$}
\[
\begin{cases}
e_1\vtr e_1=e_3+(1+\alpha_{11})e_4,\\
e_1\vtr e_2=(1+\alpha_{12})e_4,\\
e_1\vtr e_3=\alpha_{13}e_4,\\
e_2\vtr e_1=(-1+\alpha_{21})e_4,\\
e_2\vtr e_2=\alpha_{22}e_4,\\
e_2\vtr e_3=\alpha_{23}e_4,\\
e_3\vtr e_1=\alpha_{31}e_4,\\
e_3\vtr e_2=\alpha_{32}e_4,\\
e_3\vtr e_3=(1+\alpha_{33})e_4,
\end{cases}
\qquad
\begin{cases}
e_1\vtl e_1=-e_3-\alpha_{11}e_4,\\
e_1\vtl e_2=-\alpha_{12}e_4,\\
e_1\vtl e_3=-\alpha_{13}e_4,\\
e_2\vtl e_1=-\alpha_{21}e_4,\\
e_2\vtl e_2=-\alpha_{22}e_4,\\
e_2\vtl e_3=-\alpha_{23}e_4,\\
e_3\vtl e_1=-\alpha_{31}e_4,\\
e_3\vtl e_2=-\alpha_{32}e_4,\\
e_3\vtl e_3=-\alpha_{33}e_4.
\end{cases}
\]

Considering identity \eqref{id4} for the triples $\{e_1,e_1,e_3\}$
 and $\{e_3,e_1,e_1\}$, we obtain the following restrictions on structure
constants:
\[\alpha_{33}=0,\ \alpha_{33}=-1,\]
which derives a contradiction. Hence, there is not such a case.

\textbf{Case $(As_4^{10}, \vtr, \vtl)/\langle e_4\rangle\cong AD_3^{6}$}
\[
\begin{cases}
e_1\vtr e_1=(1+\alpha_{11})e_4,\\
e_1\vtr e_2=e_3+(1+\alpha_{12})e_4,\\
e_1\vtr e_3=\alpha_{13}e_4,\\
e_2\vtr e_1=(-1+\alpha_{21})e_4,\\
e_2\vtr e_2=\alpha_{22}e_4,\\
e_2\vtr e_3=\alpha_{23}e_4,\\
e_3\vtr e_1=\alpha_{31}e_4,\\
e_3\vtr e_2=\alpha_{32}e_4,\\
e_3\vtr e_3=(1+\alpha_{33})e_4,
\end{cases}
\qquad
\begin{cases}
e_1\vtl e_1=-\alpha_{11}e_4,\\
e_1\vtl e_2=-e_3-\alpha_{12}e_4,\\
e_1\vtl e_3=-\alpha_{13}e_4,\\
e_2\vtl e_1=-\alpha_{21}e_4,\\
e_2\vtl e_2=-\alpha_{22}e_4,\\
e_2\vtl e_3=-\alpha_{23}e_4,\\
e_3\vtl e_1=-\alpha_{31}e_4,\\
e_3\vtl e_2=-\alpha_{32}e_4,\\
e_3\vtl e_3=-\alpha_{33}e_4.
\end{cases}
\]

Considering the identity \eqref{id4} for the triples $\{e_1,e_2,e_3\}$
 and $\{e_3,e_1,e_2\}$, we obtain the following restrictions, respectively:
\[\alpha_{33}=0, \ \ \alpha_{33}=-1.\]
This implies a contradiction. Hence, there is not a compatible anti-dendriform algebra structure on $As_4^{10}.$

\textbf{Case $(As_4^{10}, \vtr, \vtl)/\langle e_4\rangle\cong AD_3^{7}(\lambda)$}
\[
\begin{cases}
e_1\vtr e_1=e_3+(1+\alpha_{11})e_4,\\
e_1\vtr e_2=\lambda e_3+(1+\alpha_{12})e_4,\\
e_1\vtr e_3=\alpha_{13}e_4,\\
e_2\vtr e_1=(-1+\alpha_{21})e_4,\\
e_2\vtr e_2=e_3+\alpha_{22}e_4,\\
e_2\vtr e_3=\alpha_{23}e_4,\\
e_3\vtr e_1=\alpha_{31}e_4,\\
e_3\vtr e_2=\alpha_{32}e_4,\\
e_3\vtr e_3=(1+\alpha_{33})e_4,
\end{cases}
\qquad
\begin{cases}
e_1\vtl e_1=-e_3-\alpha_{11}e_4,\\
e_1\vtl e_2=-\lambda e_3-\alpha_{12}e_4,\\
e_1\vtl e_3=-\alpha_{13}e_4,\\
e_2\vtl e_1=-\alpha_{21}e_4,\\
e_2\vtl e_2=-e_3-\alpha_{22}e_4,\\
e_2\vtl e_3=-\alpha_{23}e_4,\\
e_3\vtl e_1=-\alpha_{31}e_4,\\
e_3\vtl e_2=-\alpha_{32}e_4,\\
e_3\vtl e_3=-\alpha_{33}e_4.
\end{cases}
\]

By consediring the identity \eqref{id4} for the triples $\{e_1,e_1,e_3\}$ and $\{e_3,e_1,e_1\}$ we get the contradiction $\alpha_{33}=0$ and $\alpha_{33}=-1$. Therefore, there is not such a case.

\end{proof}

\begin{thm}\label{theo_3.6}
Any four-dimensional complex anti-dendriform algebra associated to the algebra $As_4^{13}$ is isomorphic to the following algebra:
\[
AD_4^{32}(\alpha,\beta,\gamma): \ \begin{cases}
e_1\vtr e_1=\frac12 e_3+\alpha e_4,& e_1\vtl e_1=\frac12e_3-\alpha e_4,\\
e_1\vtr e_2=\beta e_4,& e_1\vtl e_2=-\beta e_4,\\
e_1\vtr e_3=-2e_4,& e_1\vtl e_3=e_4,\\
e_2\vtr e_1=(1+\beta)e_4,& e_2\vtl e_1=-\beta e_4,\\
e_2\vtr e_2=(1+\gamma)e_4,& e_2\vtl e_2=-\gamma e_4,\\
e_3\vtr e_1=e_4, & e_3\vtl e_1=-2e_4.
\end{cases}
\]

\end{thm}
\begin{proof}
By considering \eqref{id5} for the following triples
$$\{e_1,e_2,e_2\},\
\{e_1,e_2,e_1\},\
\{e_1,e_3,e_2\}, \
\{e_2,e_3,e_1\},\
\{e_1,e_2,e_4\},\ \{e_4,e_1,e_2\},\
\{e_1,e_1,e_4\},$$ we get
$e_1\vtl e_4=0, \ e_4\vtr e_1=0, \ e_4\vtr e_2=0, \ e_2\vtl e_4=0, \ e_4\vtr e_4=0, \ e_4\vtl e_3=0, \ e_3\vtr e_4=0$, respectively.

By $e_4e_i=e_4\vtr e_i+e_4\vtl e_i=0$ and $e_ie_4=e_i\vtr e_4+e_i\vtl e_4=0$ these imply $e_4\vtl e_i=0$ and $e_i\vtr e_4=0$. Then it is easy to see that $\langle e_4\rangle$ is the center of the algebras $(As_4^{13}, \cdot)$ and $(As_4^{13}, \vtr, \vtl)$. If we take the algebra $As_4^{13}/\langle e_4\rangle,$ this algebra is a three-dimensional associative algebra $As^3_3$ with multiplication $e_1e_1=e_3$, and we know that any three-dimensional complex anti-dendriform algebra associated with the algebra $As_3^3$ is isomorphic to one of the pairwise non-isomorphic algebras $AD_3^{18}(\alpha)-AD_3^{23}.$

According to Proposition \ref{compatible} one can write

\textbf{Case $(As_4^{13}, \vtr, \vtl)/\langle e_4\rangle\cong AD_3^{18}(\alpha)$}
\[
\begin{cases}
e_1\vtr e_1=\alpha e_3+\alpha_{11}e_4,\\
e_1\vtr e_2=\alpha_{12}e_4,\\
e_1\vtr e_3=(-1+\alpha_{13})e_4,\\
e_2\vtr e_1=(1+\alpha_{21})e_4,\\
e_2\vtr e_2=(1+\alpha_{22})e_4,\\
e_2\vtr e_3=\alpha_{23}e_4,\\
e_3\vtr e_1=(-1+\alpha_{31})e_4,\\
e_3\vtr e_2=\alpha_{32}e_4,\\
e_3\vtr e_3=\alpha_{33}e_4,
\end{cases}
\qquad
\begin{cases}
e_1\vtl e_1=(1-\alpha)e_3-\alpha_{11}e_4,\\
e_1\vtl e_2=-\alpha_{12}e_4,\\
e_1\vtl e_3=-\alpha_{13}e_4,\\
e_2\vtl e_1=-\alpha_{21}e_4,\\
e_2\vtl e_2=-\alpha_{22}e_4,\\
e_2\vtl e_3=-\alpha_{23}e_4,\\
e_3\vtl e_1=-\alpha_{31}e_4,\\
e_3\vtl e_2=-\alpha_{32}e_4,\\
e_3\vtl e_3=-\alpha_{33}e_4.
\end{cases}
\]

Considering identities \eqref{id1}-\eqref{id4} for the triple $\{e_1,e_1,e_1\}$, we get the following system of equations:
\[
\begin{cases}
(-1+\alpha_{13})(-1+\alpha)-\alpha_{31}\alpha=0,\\
(-1+\alpha_{31})+(-1+\alpha_{13})\alpha=0,\\
-\alpha_{13}+(-1+\alpha_{13}) \alpha=0, \\
-\alpha_{31}(-1+\alpha)+(-1+\alpha_{13})\alpha=0.
\end{cases}
\]
When $\alpha\neq\frac12$, we get a contradiction. Hence, we can assume $\alpha=\frac12$. Then the system has only one solution $\alpha=\frac12, \ \alpha_{13}=-1, \ \alpha_{31}=2$.

By considering \eqref{id2} for the triples $\{e_1,e_1,e_2\}, \ \{e_1,e_1,e_3\}$ and $\{e_2,e_1,e_1\}$, we get restrictions
$$\alpha_{32}=\alpha_{33}=\alpha_{23}=0.$$

Thus, the multiplication table of the algebra has the following form:

\[
\begin{cases}
e_1\vtr e_1=\frac12 e_3+\alpha_{11}e_4,\\
e_1\vtr e_2=\alpha_{12}e_4,\\
e_1\vtr e_3=-2e_4,\\
e_2\vtr e_1=(1+\alpha_{21})e_4,\\
e_2\vtr e_2=(1+\alpha_{22})e_4,\\
e_3\vtr e_1=e_4,
\end{cases}
\qquad
\begin{cases}
e_1\vtl e_1=\frac12e_3-\alpha_{11}e_4,\\
e_1\vtl e_2=-\alpha_{12}e_4,\\
e_1\vtl e_3=e_4,\\
e_2\vtl e_1=-\alpha_{21}e_4,\\
e_2\vtl e_2=-\alpha_{22}e_4,\\
e_3\vtl e_1=-2e_4.
\end{cases}
\]

Let us consider the general change of the generators of basis:
\[\varphi(e_1)=e_1+ae_2+be_3+ce_4,\ \varphi(e_2)=e_2+ae_3+de_4,\ \varphi(e_3)=e_3+(a^2+a-2b)e_4,\ \varphi(e_4)=e_4.\]

We express the new basis elements $\{\varphi(e_1), \varphi(e_2),\varphi(e_3), \varphi(e_4)\}$  via the basis elements $\{e_1, e_2, e_3, e_4\}.$  By verifying all the multiplications of the algebra in the new basis we obtain the relations between the parameters $\{\alpha'_{11}, \alpha'_{12},\alpha'_{21}, \alpha'_{22}\}$ and $\{\alpha_{11}, \alpha_{12}, \alpha_{21}, \alpha_{22}\}$:
$$
\begin{array}{lll}
\varphi(e_1)\vtr \varphi(e_1)=\frac{1}{2}\varphi(e_3)+\alpha^\prime_{11}\varphi(e_4) & \Rightarrow & \alpha^\prime_{11}=\alpha_{11}+a(\frac12+\alpha_{12}+\alpha_{21})+a^2(\frac{1}{2}+\alpha_{22}),\\[1mm]
\varphi(e_1)\vtr \varphi(e_2)=\alpha^\prime_{12}\varphi(e_4) & \Rightarrow & \alpha^\prime_{12}=\alpha_{12}+a(\alpha_{22}-1),\\[1mm]
\varphi(e_2)\vtr \varphi(e_1)=(1+\alpha^\prime_{21})\varphi(e_4) & \Rightarrow & \alpha^\prime_{21}=\alpha_{21}+a(\alpha_{22}+2),\\[1mm]
\varphi(e_2)\vtr \varphi(e_2)=(1+\alpha^\prime_{22})\varphi(e_4) & \Rightarrow & \alpha^\prime_{22}=\alpha_{22}.\\[1mm]
\end{array}
$$
If choosing $a=\frac{\alpha_{12}-\alpha_{21}}{3},$ then we can reduce $\alpha^\prime_{12}=\alpha^\prime_{21}$ we obtain the algebra $AD_4^{51}(\alpha,\beta,\gamma)$.

Now if $(As_4^{13}, \vtr, \vtl)/\langle e_4\rangle$ ante-dendriform algebra is isomorphic to one of the algebras $AD_3^{19}-AD_3^{23}$, then we will proof there is not a four-dimensional complex anti-dendriform algebra associated to the associative algebra $As_4^{13}.$

\textbf{Case $(As_4^{13}, \vtr, \vtl)/\langle e_4\rangle\cong AD_3^{19}$}
\[
\begin{cases}
e_1\vtr e_1=\alpha_{11}e_4,\\
e_1\vtr e_2=\alpha_{12}e_4,\\
e_1\vtr e_3=(-1+\alpha_{13})e_4,\\
e_2\vtr e_1=e_3+(1+\alpha_{21})e_4,\\
e_2\vtr e_2=(1+\alpha_{22})e_4,\\
e_2\vtr e_3=\alpha_{23}e_4,\\
e_3\vtr e_1=(-1+\alpha_{31})e_4,\\
e_3\vtr e_2=\alpha_{32}e_4,\\
e_3\vtr e_3=\alpha_{33}e_4,
\end{cases}
\qquad
\begin{cases}
e_1\vtl e_1=e_3-\alpha_{11}e_4,\\
e_1\vtl e_2=-\alpha_{12}e_4,\\
e_1\vtl e_3=-\alpha_{13}e_4,\\
e_2\vtl e_1=-e_3-\alpha_{21}e_4,\\
e_2\vtl e_2=-\alpha_{22}e_4,\\
e_2\vtl e_3=-\alpha_{23}e_4,\\
e_3\vtl e_1=-\alpha_{31}e_4,\\
e_3\vtl e_2=-\alpha_{32}e_4,\\
e_3\vtl e_3=-\alpha_{33}e_4.
\end{cases}
\]

Considering the identity \eqref{id3} for the triples $\{e_1,e_1,e_1\}$
 and $\{e_1,e_2,e_1\}$, we obtain the following restrictions, respectively:
\[\alpha_{13}=0, \ \ \alpha_{13}=1.\]
This implies a contradiction. Hence, there is not a compatible anti-dendriform algebra structure on $As_4^{13}.$

\textbf{Case $(As_4^{13}, \vtr, \vtl)/\langle e_4\rangle\cong AD_3^{20}(\alpha)$}
\[
\begin{cases}
e_1\vtr e_1=\alpha e_3+\alpha_{11}e_4,\\
e_1\vtr e_2=e_3+\alpha_{12}e_4,\\
e_1\vtr e_3=(-1+\alpha_{13})e_4,\\
e_2\vtr e_1=-e_3+(1+\alpha_{21})e_4,\\
e_2\vtr e_2=(1+\alpha_{22})e_4,\\
e_2\vtr e_3=\alpha_{23}e_4,\\
e_3\vtr e_1=(-1+\alpha_{31})e_4,\\
e_3\vtr e_2=\alpha_{32}e_4,\\
e_3\vtr e_3=\alpha_{33}e_4,
\end{cases}
\qquad
\begin{cases}
e_1\vtl e_1=(1-\alpha)e_3-\alpha_{11}e_4,\\
e_1\vtl e_2=-e_3-\alpha_{12}e_4,\\
e_1\vtl e_3=-\alpha_{13}e_4,\\
e_2\vtl e_1=e_3-\alpha_{21}e_4,\\
e_2\vtl e_2=-\alpha_{22}e_4,\\
e_2\vtl e_3=-\alpha_{23}e_4,\\
e_3\vtl e_1=-\alpha_{31}e_4,\\
e_3\vtl e_2=-\alpha_{32}e_4,\\
e_3\vtl e_3=-\alpha_{33}e_4.
\end{cases}
\]

Consider \eqref{id3} for the triples $\{e_1,e_1,e_2\}$ and $\{e_1,e_1,e_1\}$ we obtain $\alpha_{13}=1$ and $\alpha_{13}=0$, which derives a contradiction. Hence, there is not such a case.

\textbf{Case $(As_4^{13}, \vtr, \vtl)/\langle e_4\rangle\cong AD_3^{21}(\alpha)$}
\[
\begin{cases}
e_1\vtr e_1=\alpha_{11}e_4,\\
e_1\vtr e_2=e_3+\alpha_{12}e_4,\\
e_1\vtr e_3=(-1+\alpha_{13})e_4,\\
e_2\vtr e_1=\alpha e_3+(1+\alpha_{21})e_4,\\
e_2\vtr e_2=(1+\alpha_{22})e_4,\\
e_2\vtr e_3=\alpha_{23}e_4,\\
e_3\vtr e_1=(-1+\alpha_{31})e_4,\\
e_3\vtr e_2=\alpha_{32}e_4,\\
e_3\vtr e_3=\alpha_{33}e_4,
\end{cases}
\qquad
\begin{cases}
e_1\vtl e_1=e_3-\alpha_{11}e_4,\\
e_1\vtl e_2=-e_3-\alpha_{12}e_4,\\
e_1\vtl e_3=-\alpha_{13}e_4,\\
e_2\vtl e_1=-\alpha e_3-\alpha_{21}e_4,\\
e_2\vtl e_2=-\alpha_{22}e_4,\\
e_2\vtl e_3=-\alpha_{23}e_4,\\
e_3\vtl e_1=-\alpha_{31}e_4,\\
e_3\vtl e_2=-\alpha_{32}e_4,\\
e_3\vtl e_3=-\alpha_{33}e_4.
\end{cases}
\]
where $\alpha\neq-1$.

By consediring the identity \eqref{id3} for the triples $\{e_1,e_1,e_1\}$ and $\{e_1,e_1,e_2\}$ we get the contradiction $\alpha_{13}=0$ and $\alpha_{13}=1$. Therefore, there is not such a case.

\textbf{Case $(As_4^{13}, \vtr, \vtl)/\langle e_4\rangle\cong AD_3^{22}(\alpha, \beta)$}
\[
\begin{cases}
e_1\vtr e_1=\alpha e_3+\alpha_{11}e_4,\\
e_1\vtr e_2=\alpha_{12}e_4,\\
e_1\vtr e_3=(-1+\alpha_{13})e_4,\\
e_2\vtr e_1=\beta e_3+(1+\alpha_{21})e_4,\\
e_2\vtr e_2=e_3+(1+\alpha_{22})e_4,\\
e_2\vtr e_3=\alpha_{23}e_4,\\
e_3\vtr e_1=(-1+\alpha_{31})e_4,\\
e_3\vtr e_2=\alpha_{32}e_4,\\
e_3\vtr e_3=\alpha_{33}e_4,
\end{cases}
\qquad
\begin{cases}
e_1\vtl e_1=(1-\alpha)e_3-\alpha_{11}e_4,\\
e_1\vtl e_2=-\alpha_{12}e_4,\\
e_1\vtl e_3=-\alpha_{13}e_4,\\
e_2\vtl e_1=-\beta e_3-\alpha_{21}e_4,\\
e_2\vtl e_2=-e_3-\alpha_{22}e_4,\\
e_2\vtl e_3=-\alpha_{23}e_4,\\
e_3\vtl e_1=-\alpha_{31}e_4,\\
e_3\vtl e_2=-\alpha_{32}e_4,\\
e_3\vtl e_3=-\alpha_{33}e_4.
\end{cases}
\]

By considering \eqref{id3} for the triples $\{e_1,e_2,e_2\}$ and $\{e_1,e_1,e_1\}$ we get
$$\alpha_{13}=1,\ \alpha_{13}=0.$$
This implies a contradiction. Hence, it means that there is no anti-dendriform algebra associated to
the algebra $As_4^{13}$ and satisfying the condition $(As_4^{13}, \vtr, \vtl)/\langle e_4\rangle\cong AD_3^{22}(\alpha,\beta)$.

\textbf{Case $(As_4^{13}, \vtr, \vtl)/\langle e_4\rangle\cong AD_3^{23}$}
\[
\begin{cases}
e_1\vtr e_1=e_2+\alpha_{11}e_4,\\
e_1\vtr e_2=\alpha_{12}e_4,\\
e_1\vtr e_3=(-1+\alpha_{13})e_4,\\
e_2\vtr e_1=(1+\alpha_{21})e_4,\\
e_2\vtr e_2=(1+\alpha_{22})e_4,\\
e_2\vtr e_3=\alpha_{23}e_4,\\
e_3\vtr e_1=(-1+\alpha_{31})e_4,\\
e_3\vtr e_2=\alpha_{32}e_4,\\
e_3\vtr e_3=\alpha_{33}e_4,
\end{cases}
\qquad
\begin{cases}
e_1\vtl e_1=-e_2+e_3-\alpha_{11}e_4,\\
e_1\vtl e_2=-\alpha_{12}e_4,\\
e_1\vtl e_3=-\alpha_{13}e_4,\\
e_2\vtl e_1=-\alpha_{21}e_4,\\
e_2\vtl e_2=-\alpha_{22}e_4,\\
e_2\vtl e_3=-\alpha_{23}e_4,\\
e_3\vtl e_1=-\alpha_{31}e_4,\\
e_3\vtl e_2=-\alpha_{32}e_4,\\
e_3\vtl e_3=-\alpha_{33}e_4.
\end{cases}
\]

By considering \eqref{id4} for the triples $\{e_3,e_1,e_1\}$, $\{e_2,e_1,e_1\}$ and $\{e_1,e_1,e_2\}$, we reduce
$$\alpha_{32}=0, \ \alpha_{22}=-1,\ 0=-1,$$
which is a contradiction. Hence, there is not such a case.

\end{proof}

\begin{thm}\label{theo_3.7}
Any four-dimensional complex anti-dendriform algebra associated to the algebra $As_4^{14}$ is isomorphic to one of the following pairwise non-isomorphic algebras:

$AD_4^{33}: \ e_1\vtr e_2=e_4, \ e_3\vtr e_1=e_4, \ e_1\vtl e_3=e_4, \ e_2\vtl e_1=-e_4, \ e_2\vtl e_2=e_4;$

$AD_4^{34}: \ \begin{cases}e_1\vtr e_2=e_4, \ e_2\vtr e_1=-e_4, \ e_3\vtr e_3=e_4,\\
e_1\vtl e_3=e_4, \ e_2\vtl e_2=e_4, \ e_3\vtl e_1=e_4, \ e_3\vtl e_3=-e_4;
\end{cases}$

$AD_4^{35}: \ \begin{cases}e_1\vtr e_1=e_4, \ e_1\vtr e_2=e_4, \ e_2\vtr e_1=-e_4, \ e_3\vtr e_3=e_4,\\
e_1\vtl e_1=-e_4, \ e_1\vtl e_3=e_4, \ e_2\vtl e_2=e_4, \ e_3\vtl e_1=e_4, \ e_3\vtl e_3=-e_4;
\end{cases}$

$AD_4^{36}: \ e_1\vtr e_2=e_4, \ e_1\vtr e_3=e_4, \ e_2\vtr e_1=-e_4, \ e_2\vtr e_2=e_4, \ e_3\vtr e_1=e_4;$

$AD_4^{37}(\alpha): \ \begin{cases}e_1\vtr e_1=e_4,\ e_1\vtr e_2=\alpha e_4, \ e_2\vtr e_1=-\alpha e_4, \ e_2\vtr e_2=e_4, \ e_3\vtr e_3=e_4, \ \alpha\in \mathbb{C}\\
e_1\vtl e_1=-e_4,\ e_1\vtl e_2=(1-\alpha) e_4, \ e_1\vtl e_3=e_4, \\ e_2\vtl e_1=(\alpha-1) e_4, \ e_3\vtl e_1=e_4, \ e_3\vtl e_3=-e_4;
\end{cases}$

$AD_4^{38}(\alpha): \ \begin{cases}
e_1\vtr e_2=e_3+e_4,\ e_2\vtr e_1=-e_3+(-1+\alpha )e_4,\ e_2\vtr e_2=\frac12e_4,\ e_3\vtr e_1=e_4,\\
e_1\vtl e_2=-e_3,\  e_1\vtl e_3=e_4,\ e_2\vtl e_1=e_3-\alpha e_4,\ e_2\vtl e_2=\frac12e_4, \ \alpha\neq0;
\end{cases}$

$AD_4^{39}(\alpha,\beta): \ \begin{cases}
e_1\vtr e_1=\alpha e_4,\ e_1\vtr e_2=e_3+e_4,\ e_2\vtr e_1=-e_3-e_4,\\ e_2\vtr e_2=(1+\beta)e_4,\ e_3\vtr e_1=e_4,\\
e_1\vtl e_1=-\alpha e_4,\ e_1\vtl e_2=-e_3,\ e_1\vtl e_3=e_4,\ e_2\vtl e_1=e_3,\ e_2\vtl e_2=-\beta e_4;
\end{cases}$

$AD_4^{40}(\alpha,\beta): \ \begin{cases}
e_1\vtr e_1=e_3,\ e_1\vtr e_2=(1+\alpha)e_4,\ e_2\vtr e_1=(-1+\alpha)e_4,\\ e_2\vtr e_2=(1+\beta)e_4,\ e_3\vtr e_1=e_4, \\
e_1\vtl e_1=-e_3,\ e_1\vtl e_2=-\alpha e_4,\ e_1\vtl e_3=e_4,\ e_2\vtl e_1=-\alpha e_4,\ e_2\vtl e_2=-\beta e_4;\end{cases}$

$AD_4^{41}(\alpha,\beta,\gamma): \ \begin{cases}
e_1\vtr e_1=\alpha e_4,\ e_1\vtr e_2=e_3+e_4,\ e_2\vtr e_1=(-1+\beta)e_4,\\ e_2\vtr e_2=(1+\gamma)e_4,\ e_3\vtr e_1=e_4,\\
e_1\vtl e_1=-\alpha e_4,\ e_1\vtl e_2=-e_3,\ e_1\vtl e_3=e_4,\ e_2\vtl e_1=-\beta e_4,\ e_2\vtl e_2=-\gamma e_4;
\end{cases}
$

$
AD_4^{42}[\lambda](\alpha,\beta,\gamma): \begin{cases}
e_1\vtr e_1=e_3, e_1\vtr e_2=\lambda e_3+(1+\alpha)e_4, e_2\vtr e_1=(-1+\beta)e_4,\\ e_2\vtr e_2=e_3+(1+\gamma)e_4, e_3\vtr e_1=e_4,\
e_1\vtl e_1=-e_3,\ e_1\vtl e_2=-\lambda e_3-\alpha e_4,\\ e_1\vtl e_3=e_4,\ e_2\vtl e_1=-\beta e_4,\ e_2\vtl e_2=-e_3-\gamma e_4.
\end{cases}
$
where $AD_4^{38}(0)=AD_4^{39}(0,-\frac12).$
\end{thm}
\begin{proof}

By considering \eqref{id5} for the following triples
$$\{e_3,e_1,e_1\},\
\{e_1,e_2,e_2\},\
\{e_1,e_3,e_2\}, \
\{e_2,e_3,e_1\},\
\{e_2,e_2,e_3\},\ \{e_3,e_2,e_2\},\
\{e_2,e_2,e_4\},$$ we get
$e_4\vtr e_1=0, \ e_1\vtl e_4=0, \ e_4\vtr e_2=0, \ e_2\vtl e_4=0, \ e_4\vtr e_3=0, \ e_3\vtl e_4=0, \ e_4\vtr e_4=0,$ respectively.

By $e_4e_i=e_4\vtr e_i+e_4\vtl e_i=0$ and $e_ie_4=e_i\vtr e_4+e_i\vtl e_4=0$ these imply $e_4\vtl e_i=0$ and $e_i\vtr e_4=0$. Then it is easy to see that $\langle e_4\rangle$ is the center of the algebras $(As_4^{14}, \cdot)$ and $(As_4^{14}, \vtr, \vtl)$. If we take the algebra $As_4^{14}/\langle e_4\rangle,$ this algebra is a three-dimensional associative Abelian algebra, and we know that any three-dimensional complex anti-dendriform algebra associated with the abelian algebra is isomorphic to one of the pairwise non-isomorphic algebras $AD_3^3-AD_3^7(\lambda).$

According to Proposition \ref{compatible} one can write

\textbf{Case $(As_4^{14}, \vtr, \vtl)/\langle e_4\rangle\cong AD_3^{3}$}
\[
\begin{cases}
e_1\vtr e_1=\alpha_{11}e_4,\\
e_1\vtr e_2=(1+\alpha_{12})e_4,\\
e_1\vtr e_3=(1+\alpha_{13})e_4,\\
e_2\vtr e_1=(-1+\alpha_{21})e_4,\\
e_2\vtr e_2=(1+\alpha_{22})e_4,\\
e_2\vtr e_3=\alpha_{23}e_4,\\
e_3\vtr e_1=(1+\alpha_{31})e_4,\\
e_3\vtr e_2=\alpha_{32}e_4,\\
e_3\vtr e_3=\alpha_{33}e_4,
\end{cases}
\qquad
\begin{cases}
e_1\vtl e_1=-\alpha_{11}e_4,\\
e_1\vtl e_2=-\alpha_{12}e_4,\\
e_1\vtl e_3=-\alpha_{13}e_4,\\
e_2\vtl e_1=-\alpha_{21}e_4,\\
e_2\vtl e_2=-\alpha_{22}e_4,\\
e_2\vtl e_3=-\alpha_{23}e_4,\\
e_3\vtl e_1=-\alpha_{31}e_4,\\
e_3\vtl e_2=-\alpha_{32}e_4,\\
e_3\vtl e_3=-\alpha_{33}e_4.
\end{cases}
\]

It is not difficult to see that $(As_4^{14}, \vtr, \vtl)$ are 2-nilpotent and have three generator elements. Thus we have $(x\vtr y)\vtr z=x\vtr (y\vtr z)=0$ which implies that $(As_4^{14}, \vtr)$ is also associative 2-nilpotent algebra and has three generator elements. According to Theorem \ref{table} there are five non-isomorphic four-dimensional indecomposable associative 2-nilpotent and the three generated algebras. Hence, we get $AD_4^{33}-AD_4^{37}(\alpha)$ algebras.

\textbf{Case $(As_4^{14}, \vtr, \vtl)/\langle e_4\rangle\cong AD_3^{4}$}
\[
\begin{cases}
e_1\vtr e_1=\alpha_{11}e_4,\\
e_1\vtr e_2=e_3+(1+\alpha_{12})e_4,\\
e_1\vtr e_3=(1+\alpha_{13})e_4,\\
e_2\vtr e_1=-e_3+(-1+\alpha_{21})e_4,\\
e_2\vtr e_2=(1+\alpha_{22})e_4,\\
e_2\vtr e_3=\alpha_{23}e_4,\\
e_3\vtr e_1=(1+\alpha_{31})e_4,\\
e_3\vtr e_2=\alpha_{32}e_4,\\
e_3\vtr e_3=\alpha_{33}e_4,
\end{cases}
\qquad
\begin{cases}
e_1\vtl e_1=-\alpha_{11}e_4,\\
e_1\vtl e_2=-e_3-\alpha_{12}e_4,\\
e_1\vtl e_3=-\alpha_{13}e_4,\\
e_2\vtl e_1=e_3-\alpha_{21}e_4,\\
e_2\vtl e_2=-\alpha_{22}e_4,\\
e_2\vtl e_3=-\alpha_{23}e_4,\\
e_3\vtl e_1=-\alpha_{31}e_4,\\
e_3\vtl e_2=-\alpha_{32}e_4,\\
e_3\vtl e_3=-\alpha_{33}e_4.
\end{cases}
\]

Considering \eqref{id4} for the triples $\{e_1,e_1,e_2\}$, $\{e_2,e_2,e_1\}$
 and $\{e_1,e_2,e_i\}, \ i=1,2,3$, we can obtain
 $$\alpha_{13}=-1, \ \alpha_{23}=0, \ \alpha_{3i}=0, \ i=1,2,3.$$

By the basis change $e'_3=e_3+\alpha_{12}e_4$ the multiplication table of the algebra has the following form:
\[
\begin{cases}
e_1\vtr e_1=\alpha_{11}e_4,\\
e_1\vtr e_2=e_3+e_4,\\
e_2\vtr e_1=-e_3+(-1+\alpha_{21})e_4,\\
e_2\vtr e_2=(1+\alpha_{22})e_4,\\
e_3\vtr e_1=e_4,
\end{cases}
\qquad
\begin{cases}
e_1\vtl e_1=-\alpha_{11}e_4,\\
e_1\vtl e_2=-e_3,\\
e_1\vtl e_3=e_4,\\
e_2\vtl e_1=e_3-\alpha_{21}e_4,\\
e_2\vtl e_2=-\alpha_{22}e_4.
\end{cases}
\]

Let us consider the general change of the generators of basis:
\[\varphi(e_1)=ae_1+be_2-\frac{b^2}{2a}e_3+ce_4,\ \varphi(e_2)=ae_2-be_3+de_4,\ \varphi(e_3)=ae_3+ee_4,\ \varphi(e_4)=a^2e_4,\]
where $a\neq0$.

We express the new basis elements $\{\varphi(e_1), \varphi(e_2),\varphi(e_3), \varphi(e_4)\}$  via the basis elements $\{e_1, e_2, e_3, e_4\}.$  By verifying all the multiplications of the algebra in the new basis we obtain the relations between the parameters $\{\alpha'_{11}, \alpha'_{21}, \alpha'_{22}\}$ and $\{\alpha_{11}, \alpha_{21}, \alpha_{22}\}$:

$$
\begin{array}{lll}
\varphi(e_1)\vtr \varphi(e_2)=\varphi(e_3)+\varphi(e_4) & \Rightarrow & a^2=a, \  b(1+\alpha_{22})=e,\\[1mm]
\varphi(e_1)\vtr \varphi(e_1)=\alpha^\prime_{11}\varphi(e_4) & \Rightarrow & \alpha^\prime_{11}=\alpha_{11}+b\alpha_{21}+b^2(\frac{1}{2}+\alpha_{22}),\\[1mm]
\varphi(e_1)\vtr \varphi(e_2)=-\varphi(e_3)+(-1+\alpha^\prime_{21})\varphi(e_4) & \Rightarrow & \alpha^\prime_{21}=\alpha_{21}+b(1+2\alpha_{22}),\\[1mm]
\varphi(e_2)\vtr \varphi(e_2)=(1+\alpha^\prime_{22})\varphi(e_4) & \Rightarrow & \alpha^\prime_{22}=\alpha_{22}.\\[1mm]
\end{array}
$$

We have the following cases:
\begin{itemize}
\item[(a)] Let $\alpha_{22}=-\frac12$. Then

If $\alpha_{21}=0$, we obtain the algebra $AD_4^{39}(\alpha, -\frac12)$.

If $\alpha_{21}\neq0$, then by choosing $b=-\frac{\alpha_{11}}{\alpha_{21}}$ we obtain $\alpha^\prime_{11}=0$ and we get the algebra $AD_4^{38}(\alpha), \ \alpha\neq0$.
\item[(b)] Let $\alpha_{22}\neq-\frac12$. Then choosing $b=-\frac{\alpha_{21}}{1+2\alpha_{22}}$ we derive $\alpha^\prime_{21}=0$ and  we obtain the algebra $AD_4^{39}(\alpha,\beta), \ \beta\neq-\frac12$.
\end{itemize}

\textbf{Case $(As_4^{14}, \vtr, \vtl)/\langle e_4\rangle\cong AD_3^{5}$}
\[
\begin{cases}
e_1\vtr e_1=e_3+\alpha_{11}e_4,\\
e_1\vtr e_2=(1+\alpha_{12})e_4,\\
e_1\vtr e_3=(1+\alpha_{13})e_4,\\
e_2\vtr e_1=(-1+\alpha_{21})e_4,\\
e_2\vtr e_2=(1+\alpha_{22})e_4,\\
e_2\vtr e_3=\alpha_{23}e_4,\\
e_3\vtr e_1=(1+\alpha_{31})e_4,\\
e_3\vtr e_2=\alpha_{32}e_4,\\
e_3\vtr e_3=\alpha_{33}e_4,
\end{cases}
\qquad
\begin{cases}
e_1\vtl e_1=-e_3-\alpha_{11}e_4,\\
e_1\vtl e_2=-\alpha_{12}e_4,\\
e_1\vtl e_3=-\alpha_{13}e_4,\\
e_2\vtl e_1=-\alpha_{21}e_4,\\
e_2\vtl e_2=-\alpha_{22}e_4,\\
e_2\vtl e_3=-\alpha_{23}e_4,\\
e_3\vtl e_1=-\alpha_{31}e_4,\\
e_3\vtl e_2=-\alpha_{32}e_4,\\
e_3\vtl e_3=-\alpha_{33}e_4.
\end{cases}
\]

Considering \eqref{id2} for the triples $\{e_i,e_1,e_1\}, \ i=1,2,3$, we obtain $\alpha_{13}=-1, \ \alpha_{23}=\alpha_{33}=0$. From \eqref{id4} for the triples $\{e_1,e_1,e_i\}, \ i=1,2$, we get $\alpha_{31}=\alpha_{32}=0$. By the basis change $e'_3=e_3+\alpha_{11}e_4$ we get $\alpha_{11}=0$.

Thus, we have the algebra
\[
\begin{cases}
e_1\vtr e_1=e_3,\\
e_1\vtr e_2=(1+\alpha_{12})e_4,\\
e_2\vtr e_1=(-1+\alpha_{21})e_4,\\
e_2\vtr e_2=(1+\alpha_{22})e_4,\\
e_3\vtr e_1=e_4,
\end{cases}
\qquad
\begin{cases}
e_1\vtl e_1=-e_3,\\
e_1\vtl e_2=-\alpha_{12}e_4,\\
e_1\vtl e_3=e_4,\\
e_2\vtl e_1=-\alpha_{21}e_4,\\
e_2\vtl e_2=-\alpha_{22}e_4.
\end{cases}
\]

Let us consider the general change and by verifying all the multiplications of the algebra in the new basis we obtain the relations between the parameters $\{\alpha'_{12}, \alpha'_{21}, \alpha'_{22}\}$ and $\{\alpha_{12}, \alpha_{21}, \alpha_{22}\}$:
$$
\begin{array}{lll}
\varphi(e_1)\vtr \varphi(e_1)=\varphi(e_3) & \Rightarrow & a^2=a, \  b(\alpha_{12}+\alpha_{21})+b^2(\frac12+\alpha_{22})=e,\\[1mm]
\varphi(e_1)\vtr \varphi(e_2)=(1+\alpha^\prime_{12})\varphi(e_4) & \Rightarrow & \alpha^\prime_{12}=\alpha_{12}+b(1+\alpha_{22}),\\[1mm]
\varphi(e_2)\vtr \varphi(e_1)=(-1+\alpha^\prime_{21})\varphi(e_4) & \Rightarrow & \alpha^\prime_{21}=\alpha_{21}+b\alpha_{22},\\[1mm]
\varphi(e_2)\vtr \varphi(e_2)=(1+\alpha^\prime_{22})\varphi(e_4) & \Rightarrow & \alpha^\prime_{22}=\alpha_{22}.\\[1mm]
\end{array}
$$
If choosing $b=\alpha_{21}-\alpha_{12},$ then we can reduce $\alpha^\prime_{12}=\alpha^\prime_{21}$ we obtain the algebra $AD_4^{40}(\alpha,\beta)$.

\textbf{Case $(As_4^{14}, \vtr, \vtl)/\langle e_4\rangle\cong AD_3^{6}$}
\[
\begin{cases}
e_1\vtr e_1=\alpha_{11}e_4,\\
e_1\vtr e_2=e_3+(1+\alpha_{12})e_4,\\
e_1\vtr e_3=(1+\alpha_{13})e_4,\\
e_2\vtr e_1=(-1+\alpha_{21})e_4,\\
e_2\vtr e_2=(1+\alpha_{22})e_4,\\
e_2\vtr e_3=\alpha_{23}e_4,\\
e_3\vtr e_1=(1+\alpha_{31})e_4,\\
e_3\vtr e_2=\alpha_{32}e_4,\\
e_3\vtr e_3=\alpha_{33}e_4,
\end{cases}
\qquad
\begin{cases}
e_1\vtl e_1=-\alpha_{11}e_4,\\
e_1\vtl e_2=-e_3-\alpha_{12}e_4,\\
e_1\vtl e_3=-\alpha_{13}e_4,\\
e_2\vtl e_1=-\alpha_{21}e_4,\\
e_2\vtl e_2=-\alpha_{22}e_4,\\
e_2\vtl e_3=-\alpha_{23}e_4,\\
e_3\vtl e_1=-\alpha_{31}e_4,\\
e_3\vtl e_2=-\alpha_{32}e_4,\\
e_3\vtl e_3=-\alpha_{33}e_4.
\end{cases}
\]

By the basis change $e'_3=e_3+\alpha_{12}e_4$, we can put $\alpha_{12}=0$. Further considering \eqref{id1} for the triple $\{e_1,e_1,e_2\}$, \ $\{e_2,e_1,e_2\}$  and $\{e_1,e_2,e_i\}, i=1,2,3$, we obtain the following restrictions:
$$\alpha_{13}=-1, \ \alpha_{23}=0, \alpha_{3i}=0, \ i=1,2,3.$$

Thus, the multiplication table of the algebra has the following form:
\[
\begin{cases}
e_1\vtr e_1=\alpha_{11}e_4,\\
e_1\vtr e_2=e_3+e_4,\\
e_2\vtr e_1=(-1+\alpha_{21})e_4,\\
e_2\vtr e_2=(1+\alpha_{22})e_4,\\
e_3\vtr e_1=e_4,
\end{cases}
\qquad
\begin{cases}
e_1\vtl e_1=-\alpha_{11}e_4,\\
e_1\vtl e_2=-e_3,\\
e_1\vtl e_3=e_4,\\
e_2\vtl e_1=-\alpha_{21}e_4,\\
e_2\vtl e_2=-\alpha_{22}e_4.
\end{cases}
\]

If we consider the general change of the basis in this algebra to study the behavior of these structure constants of the algebra, then we have the following relations:
$$ \alpha'_{11}=\alpha_{11},\  \alpha'_{21}=\alpha_{21},\ \alpha'_{22}=\alpha_{22}$$
Hence, we have the algebra $AD_4^{41}(\alpha,\beta,\gamma)$.

\textbf{Case $(As_4^{14}, \vtr, \vtl)/\langle e_4\rangle\cong AD_3^{7}(\lambda)$}
\[
\begin{cases}
e_1\vtr e_1=e_3+\alpha_{11}e_4,\\
e_1\vtr e_2=\lambda e_3+(1+\alpha_{12})e_4,\\
e_1\vtr e_3=(1+\alpha_{13})e_4,\\
e_2\vtr e_1=(-1+\alpha_{21})e_4,\\
e_2\vtr e_2=e_3+(1+\alpha_{22})e_4,\\
e_2\vtr e_3=\alpha_{23}e_4,\\
e_3\vtr e_1=(1+\alpha_{31})e_4,\\
e_3\vtr e_2=\alpha_{32}e_4,\\
e_3\vtr e_3=\alpha_{33}e_4,
\end{cases}
\qquad
\begin{cases}
e_1\vtl e_1=-e_3-\alpha_{11}e_4,\\
e_1\vtl e_2=-\lambda e_3-\alpha_{12}e_4,\\
e_1\vtl e_3=-\alpha_{13}e_4,\\
e_2\vtl e_1=-\alpha_{21}e_4,\\
e_2\vtl e_2=-e_3-\alpha_{22}e_4,\\
e_2\vtl e_3=-\alpha_{23}e_4,\\
e_3\vtl e_1=-\alpha_{31}e_4,\\
e_3\vtl e_2=-\alpha_{32}e_4,\\
e_3\vtl e_3=-\alpha_{33}e_4.
\end{cases}
\]

By the basis change $e'_3=e_3+\alpha_{11}e_4$, we can get $\alpha_{11}=0$. Considering \eqref{id3} for the triple $\{e_i,e_1,e_1\}, \ i=1,2,3$, and also considering \eqref{id4} for the triple $\{e_1,e_1,e_i\}, \ i=1,2,$ we obtain the following restrictions on structure constants:
$$\alpha_{13}=-1, \ \alpha_{23}=\alpha_{33}=\alpha_{31}=\alpha_{32}=0.$$

Now we  rewrite
\[
\begin{cases}
e_1\vtr e_1=e_3,\\
e_1\vtr e_2=\lambda e_3+(1+\alpha_{12})e_4,\\
e_2\vtr e_1=(-1+\alpha_{21})e_4,\\
e_2\vtr e_2=e_3+(1+\alpha_{22})e_4,\\
e_3\vtr e_1=e_4,
\end{cases}
\qquad
\begin{cases}
e_1\vtl e_1=-e_3,\\
e_1\vtl e_2=-\lambda e_3-\alpha_{12}e_4,\\
e_1\vtl e_3=e_4,\\
e_2\vtl e_1=-\alpha_{21}e_4,\\
e_2\vtl e_2=-e_3-\alpha_{22}e_4.
\end{cases}
\]

Similarly, by considering the general change of the basis in this algebra we can show the following relations:
$$
\alpha'_{12}=\alpha_{12},\  \alpha'_{21}=\alpha_{21},\  \alpha'_{22}=\alpha_{22}.$$
Hence, we have the algebra $AD_4^{42}[\lambda](\alpha,\beta,\gamma)$
\end{proof}

\begin{thm}\label{theo_3.8}
Any four-dimensional complex anti-dendriform algebra associated to the algebra $As_4^{15}(\alpha)$ is isomorphic to one of the following pairwise non-isomorphic algebras:

$AD_4^{43}(\alpha): \ \begin{cases}e_1\vtr e_2=e_4, \  e_3\vtr e_1=e_4,\\
e_1\vtl e_1=e_4, \ e_1\vtl e_2=(\alpha-1)e_4,\ e_2\vtl e_1=-\alpha e_4, \\ e_2\vtl e_2=e_4, \ e_3\vtl e_1=-e_4, \ e_3\vtl e_3=e_4;
\end{cases}$

$AD_4^{44}(\alpha): \ \begin{cases}e_1\vtr e_2=e_4, \ e_2\vtr e_1=-e_4,\  e_3\vtr e_3=e_4,\\
e_1\vtl e_1=e_4, \  e_1\vtl e_2=(\alpha-1)e_4, \  e_2\vtl e_1=(1-\alpha) e_4, \  e_2\vtl e_2=e_4;
\end{cases}$

$AD_4^{45}(\alpha): \  \begin{cases}e_1\vtr e_1=e_4, \  e_1\vtr e_2=e_4, \  e_2\vtr e_1=-e_4, \  e_3\vtr e_3=e_4,\\
e_1\vtl e_2=(\alpha-1)e_4, \ e_2\vtl e_1=(1-\alpha)e_4, \ e_2\vtl e_2=e_4;
\end{cases}$

$AD_4^{46}(\alpha): \ \begin{cases}e_1\vtr e_2=e_4, \ e_1\vtr e_3=e_4, \ e_2\vtr e_1=-e_4, \ e_2\vtr e_2=e_4, \ e_3\vtr e_1=e_4,\\
e_1\vtl e_1=e_4, \ e_1\vtl e_2=(\alpha-1)e_4, \ e_1\vtl e_3=-e_4, \\ e_2\vtl e_1=(1-\alpha)e_4,  \ e_3\vtl e_1=-e_4, \ e_3\vtl e_3=e_4;
\end{cases}$

$AD_4^{47}(\alpha): \ e_1\vtr e_1=e_4,\ e_1\vtr e_2=\alpha e_4, \ e_2\vtr e_1=-\alpha e_4, \ e_2\vtr e_2=e_4, \ e_3\vtr e_3=e_4.$

\end{thm}
\begin{proof}
By considering \eqref{id5} for the following triples
$$\{e_1,e_3,e_3\},\
\{e_3,e_3,e_1\},\
\{e_3,e_3,e_2\}, \
\{e_2,e_3,e_3\},\
\{e_2,e_2,e_3\},\ \{e_3,e_2,e_2\},\
\{e_3,e_3,e_4\},$$ we get
$e_1\vtl e_4=0, \ e_4\vtr e_1=0, \ e_4\vtr e_2=0, \ e_2\vtl e_4=0, \ e_4\vtr e_3=0, \ e_3\vtl e_4=0, \ e_4\vtr e_4=0,$ respectively.

By $e_4e_i=e_4\vtr e_i+e_4\vtl e_i=0$ and $e_ie_4=e_i\vtr e_4+e_i\vtl e_4=0$ these imply $e_4\vtl e_i=0$ and $e_i\vtr e_4=0$. Then it is easy to see that $\langle e_4\rangle$ is the center of the algebras $(As_4^{15}(\alpha), \cdot)$ and $(As_4^{15}(\alpha), \vtr, \vtl)$. If we take the algebra $As_4^{15}(\alpha)/\langle e_4\rangle,$ this algebra is a three-dimensional associative abelian algebra, and we know that any three-dimensional complex anti-dendriform algebra associated with the abelian algebra is isomorphic to one of the pairwise non-isomorphic algebras $AD_3^3-AD_3^7(\lambda).$

According to Proposition \ref{compatible} one can write

\textbf{Case $(As_4^{15}(\alpha), \vtr, \vtl)/\langle e_4\rangle\cong AD_3^{3}$}
\[
\begin{cases}
e_1\vtr e_1=(1+\alpha_{11})e_4,\\
e_1\vtr e_2=(\alpha+\alpha_{12})e_4,\\
e_1\vtr e_3=\alpha_{13}e_4,\\
e_2\vtr e_1=(-\alpha+\alpha_{21})e_4,\\
e_2\vtr e_2=(1+\alpha_{22})e_4,\\
e_2\vtr e_3=\alpha_{23}e_4,\\
e_3\vtr e_1=\alpha_{31}e_4,\\
e_3\vtr e_2=\alpha_{32}e_4,\\
e_3\vtr e_3=(1+\alpha_{33})e_4,
\end{cases}
\qquad
\begin{cases}
e_1\vtl e_1=-\alpha_{11}e_4,\\
e_1\vtl e_2=-\alpha_{12}e_4,\\
e_1\vtl e_3=-\alpha_{13}e_4,\\
e_2\vtl e_1=-\alpha_{21}e_4,\\
e_2\vtl e_2=-\alpha_{22}e_4,\\
e_2\vtl e_3=-\alpha_{23}e_4,\\
e_3\vtl e_1=-\alpha_{31}e_4,\\
e_3\vtl e_2=-\alpha_{32}e_4,\\
e_3\vtl e_3=-\alpha_{33}e_4.
\end{cases}
\]

It is not difficult to see that $(As_4^{15}, \vtr, \vtl)$ are 2-nilpotent and have three generator elements. Thus we have $(x\vtr y)\vtr z=x\vtr (y\vtr z)=0$ which implies that $(As_4^{15}, \vtr)$ is also associative 2-nilpotent algebra and has three generator elements. According to Theorem \ref{table} there are five non-isomorphic four-dimensional indecomposable associative 2-nilpotent and the three generated algebras. Hence, we get $AD_4^{43}(\alpha)-AD_4^{47}(\alpha)$ algebras.

It is possible to consider the multiplication $\vtl$ as above. However, it is not difficult to show that the constructed algebras are isomorphic.

Now if $(As_4^{15}, \vtr, \vtl)/\langle e_4\rangle$ ante-dendriform algebra is isomorphic to one of the algebras $AD_3^{4}-AD_3^{7}(\alpha)$, then we will proof there is not a four-dimensional complex anti-dendriform algebra associated to the associative algebra $As_4^{15}.$

\textbf{Case $(As_4^{15}(\alpha), \vtr, \vtl)/\langle e_4\rangle\cong AD_3^{4}$}
\[
\begin{cases}
e_1\vtr e_1=(1+\alpha_{11})e_4,\\
e_1\vtr e_2=e_3+(\alpha+\alpha_{12})e_4,\\
e_1\vtr e_3=\alpha_{13}e_4,\\
e_2\vtr e_1=-e_3+(-\alpha+\alpha_{21})e_4,\\
e_2\vtr e_2=(1+\alpha_{22})e_4,\\
e_2\vtr e_3=\alpha_{23}e_4,\\
e_3\vtr e_1=\alpha_{31}e_4,\\
e_3\vtr e_2=\alpha_{32}e_4,\\
e_3\vtr e_3=(1+\alpha_{33})e_4,
\end{cases}
\qquad
\begin{cases}
e_1\vtl e_1=-\alpha_{11}e_4,\\
e_1\vtl e_2=-e_3-\alpha_{12}e_4,\\
e_1\vtl e_3=-\alpha_{13}e_4,\\
e_2\vtl e_1=e_3-\alpha_{21}e_4,\\
e_2\vtl e_2=-\alpha_{22}e_4,\\
e_2\vtl e_3=-\alpha_{23}e_4,\\
e_3\vtl e_1=-\alpha_{31}e_4,\\
e_3\vtl e_2=-\alpha_{32}e_4,\\
e_3\vtl e_3=-\alpha_{33}e_4.
\end{cases}
\]

Considering the identity \eqref{id4} for the triples $\{e_2,e_1,e_3\}$ and $\{e_3,e_1,e_2\}$, we obtain the following
restrictions on structure constants:
$$\alpha_{33}=0, \ \ \alpha_{33}=-1.$$
This implies a contradiction. Hence, it means that there is no anti-dendriform algebra associated to the algebra $As_4^{15}$ and satisfying the condition $(As_4^{15}(\alpha), \vtr, \vtl)/\langle e_4\rangle\cong AD_3^4$.

\textbf{Case $(As_4^{15}(\alpha), \vtr, \vtl)/\langle e_4\rangle\cong AD_3^{5}$}
\[
\begin{cases}
e_1\vtr e_1=e_3+(1+\alpha_{11})e_4,\\
e_1\vtr e_2=(\alpha+\alpha_{12})e_4,\\
e_1\vtr e_3=\alpha_{13}e_4,\\
e_2\vtr e_1=(-\alpha+\alpha_{21})e_4,\\
e_2\vtr e_2=(1+\alpha_{22})e_4,\\
e_2\vtr e_3=\alpha_{23}e_4,\\
e_3\vtr e_1=\alpha_{31}e_4,\\
e_3\vtr e_2=\alpha_{32}e_4,\\
e_3\vtr e_3=(1+\alpha_{33})e_4,
\end{cases}
\qquad
\begin{cases}
e_1\vtl e_1=-e_3-\alpha_{11}e_4,\\
e_1\vtl e_2=-\alpha_{12}e_4,\\
e_1\vtl e_3=-\alpha_{13}e_4,\\
e_2\vtl e_1=-\alpha_{21}e_4,\\
e_2\vtl e_2=-\alpha_{22}e_4,\\
e_2\vtl e_3=-\alpha_{23}e_4,\\
e_3\vtl e_1=-\alpha_{31}e_4,\\
e_3\vtl e_2=-\alpha_{32}e_4,\\
e_3\vtl e_3=-\alpha_{33}e_4.
\end{cases}
\]

By considering \eqref{id4} for the triples $\{e_1,e_1,e_3\}$ and $\{e_3,e_1,e_1\}$, we reduce
$$\alpha_{33}=0, \ \alpha_{33}=-1,$$
which is a contradiction. Hence, there is not such a case.

\textbf{Case $(As_4^{15}(\alpha), \vtr, \vtl)/\langle e_4\rangle\cong AD_3^{6}$}
\[
\begin{cases}
e_1\vtr e_1=(1+\alpha_{11})e_4,\\
e_1\vtr e_2=e_3+(\alpha+\alpha_{12})e_4,\\
e_1\vtr e_3=\alpha_{13}e_4,\\
e_2\vtr e_1=(-\alpha+\alpha_{21})e_4,\\
e_2\vtr e_2=(1+\alpha_{22})e_4,\\
e_2\vtr e_3=\alpha_{23}e_4,\\
e_3\vtr e_1=\alpha_{31}e_4,\\
e_3\vtr e_2=\alpha_{32}e_4,\\
e_3\vtr e_3=(1+\alpha_{33})e_4,
\end{cases}
\qquad
\begin{cases}
e_1\vtl e_1=-\alpha_{11}e_4,\\
e_1\vtl e_2=-e_3-\alpha_{12}e_4,\\
e_1\vtl e_3=-\alpha_{13}e_4,\\
e_2\vtl e_1=-\alpha_{21}e_4,\\
e_2\vtl e_2=-\alpha_{22}e_4,\\
e_2\vtl e_3=-\alpha_{23}e_4,\\
e_3\vtl e_1=-\alpha_{31}e_4,\\
e_3\vtl e_2=-\alpha_{32}e_4,\\
e_3\vtl e_3=-\alpha_{33}e_4.
\end{cases}
\]

By consediring the identity \eqref{id4} for the triples $\{e_1,e_2,e_3\}$ and $\{e_3,e_1,e_2\}$ we get the contradiction $\alpha_{33}=0$ and $\alpha_{33}=-1$. Therefore, there is not such a case.

\textbf{Case $(As_4^{15}(\alpha), \vtr, \vtl)/\langle e_4\rangle\cong AD_3^{7}(\lambda)$}
\[
\begin{cases}
e_1\vtr e_1=e_3+(1+\alpha_{11})e_4,\\
e_1\vtr e_2=\lambda e_3+(\alpha+\alpha_{12})e_4,\\
e_1\vtr e_3=\alpha_{13}e_4,\\
e_2\vtr e_1=(-\alpha+\alpha_{21})e_4,\\
e_2\vtr e_2=e_3+(1+\alpha_{22})e_4,\\
e_2\vtr e_3=\alpha_{23}e_4,\\
e_3\vtr e_1=\alpha_{31}e_4,\\
e_3\vtr e_2=\alpha_{32}e_4,\\
e_3\vtr e_3=(1+\alpha_{33})e_4,
\end{cases}
\qquad
\begin{cases}
e_1\vtl e_1=-e_3-\alpha_{11}e_4,\\
e_1\vtl e_2=-\lambda e_3-\alpha_{12}e_4,\\
e_1\vtl e_3=-\alpha_{13}e_4,\\
e_2\vtl e_1=-\alpha_{21}e_4,\\
e_2\vtl e_2=-e_3-\alpha_{22}e_4,\\
e_2\vtl e_3=-\alpha_{23}e_4,\\
e_3\vtl e_1=-\alpha_{31}e_4,\\
e_3\vtl e_2=-\alpha_{32}e_4,\\
e_3\vtl e_3=-\alpha_{33}e_4.
\end{cases}
\]
Consider \eqref{id4} for the triples $\{e_1,e_1,e_3\}$ and $\{e_3,e_1,e_1\}$ we obtain $\alpha_{33}=0, \ \alpha_{33}=-1,$
which implies a contradiction. Hence, there is not such a case.

\end{proof}

\end{document}